\documentclass[12pt,english,a4paper]{article}

\usepackage{amsmath}
\usepackage{amsthm}
\usepackage{amssymb}
\usepackage{dsfont}
\usepackage{stmaryrd}
\usepackage{commath}
\usepackage{mathrsfs}
\usepackage{enumitem}
\usepackage{geometry}
\usepackage{graphicx}
\usepackage{setspace}
\usepackage{caption}
\usepackage{hyperref}
\usepackage[normalem]{ulem}
\usepackage{xcolor}

\newtheorem{thm}{Theorem}
\newtheorem{lemma}[thm]{Lemma}
\newtheorem{coro}[thm]{Corollary}
\newtheorem{prop}[thm]{Proposition}
\newtheorem{rem}[thm]{Remark}
\newtheorem{OP}{Open Problem}

\newcommand{\R}{\mathbb{R}}
\newcommand{\IR}{\mathbb{R}}

\newcommand{\C}{\ensuremath{\mathbb{C}}}
\newcommand{\N}{\mathbb{N}}
\newcommand{\IN}{\mathbb{N}}

\newcommand{\eps}{\varepsilon}
\renewcommand{\phi}{\varphi}

\newcommand{\dint}{{\,\mathrm d}}
\newcommand{\ipr}[1]{\left\langle #1 \right\rangle}
\renewcommand{\l}{\langle}
\renewcommand{\r}{\rangle}
\newcommand{\PP}{\mathbb{P}}

\newcommand{\E}{\mathbb{E}}

\newcommand{\cE}{{\mathcal{E}}}

\newcommand{\Pn}{{\mathcal{P}_n}}

\newcommand{\B}{\mathcal{B}}
\newcommand{\dist}{\,{\mathrm{dist}}}

\newcommand{\info}{\Lambda}

\newcommand{\id}{\ensuremath{\mathrm{id}}}

\newcommand{\lin}{\ensuremath{\mathrm{lin}}}
\newcommand{\ran}{\ensuremath{\mathrm{ran}}}
\newcommand{\std}{\ensuremath{\mathrm{std}}}
\newcommand{\all}{\ensuremath{\mathrm{all}}}
\newcommand{\iid}{{iid}}

\DeclareMathOperator*{\argmin}{argmin}
\DeclareMathOperator*{\rad}{rad}

\title{On the power of iid information \\ for linear approximation}
\author{Mathias Sonnleitner\footnote{Faculty of Computer Science and Mathematics, University of Passau, 94032 Passau, Germany, \texttt{mathias.sonnleitner@uni-passau.de}}, Mario Ullrich\footnote{Institute of Analysis, Johannes Kepler University Linz, 4040 Linz, Austria, \texttt{mario.ullrich@jku.at}}}

\date{\today}

\begin{document}

\maketitle

\begin{abstract}
This survey is concerned with the 
power of random information
for approximation in the (deterministic) worst-case setting, 
with special emphasis on information consisting of functionals selected 
independently and identically distributed (iid) at random on a class of
admissible information functionals. 
We present a general result based on a weighted least squares method 
and derive consequences for special cases. 
Improvements are available 
if the information is ``Gaussian'' or if we consider 
iid function values for Sobolev spaces. 
We include open questions to guide future research on the power of random information in the context of information-based complexity.
\end{abstract}

\medskip

\centerline{\begin{minipage}[hc]{130mm}{
}
\end{minipage}}
\medskip

\setcounter{tocdepth}{2}
\setcounter{secnumdepth}{3}

\tableofcontents

\bigskip

\section{Introduction}
 \label{sec:intro}

This survey is oriented towards \emph{information-based complexity} (IBC)
and we refer to~\cite{No88,NoWo08,NoWo12,TWW88} 
for a more comprehensive treatment of information-based complexity.  
For an introduction to the related field of optimal recovery we refer 
to~\cite{Dev98,Fou22,Pinkus85,Te18}. 
Perhaps the most prominent display of the power of iid information, 
as we understand it, is the field of compressed sensing (or sparse recovery) which is presented 
in~\cite{Dono06} from the viewpoint of IBC. 
We will only briefly touch upon this direction as our focus is 
on \textit{linear approximation}. 

Numerical approximation, as 
considered here, 
is formally specified by two normed real vector spaces, 
say $H$ and $G$, 
of functions on a set $D$, 
a subset $F\subset H$ for which also $F\subset G$, 
and a class of \emph{admissible information} $\Lambda$ 
consisting of 
functionals on $F$. 
(We often consider $F$ to be the unit ball $B_H$ in $H$ and assume that the functionals are defined on $H$.)  
The aim is to approximate the \emph{a priori unknown} $f\in F$ 
based on~$n$ pieces of information (or measurements) 
$\ell_1(f),\dots,\ell_n(f)$ with $\ell_i\in\Lambda$ such that we can guarantee 
a small error with respect to the norm in~$G$.
In general, one does not have access to arbitrary measurements, 
which is the reason for restricting to $\Lambda$. 
Typical examples for admissible information are
\begin{itemize}
\itemsep=-1mm
   \item $\Lambda^{\all}:=H'$, i.e., all continuous linear functionals on $H$, 
	\item $\Lambda = $ ``certain expectations of the input function'', 
	\item $\Lambda = $ ``coefficients w.r.t.~a given basis, wavelets etc.'' or  
	\item $\Lambda^{\std} \,:=\, \{\delta_x \colon \delta_x(f)=f(x),\, f\in H,\,  
		x\in D\}$ \;\; (\emph{function values}). 
\end{itemize}

It is desirable to minimize the approximation error, which is achieved 
by the ``best information'' from a given class $\Lambda$.  
To make this precise, we identify information with an \emph{information mapping} 
of the form
\begin{equation}
\label{eq:info}
 N_n\colon H\to \R^n,
 \quad
 N_n(f)=\big(\ell_1(f),\hdots,\ell_n(f)\big), \quad f\in H,
\end{equation}
with  $\ell_1,\dots,\ell_n\in\Lambda$. 
(One may also consider adaptive information, i.e., $\ell_j$ may depend on the already computed $\ell_1(f),\hdots,\ell_{j-1}(f)$, 
but we do not treat this here.) Any approximation method (or \textit{algorithm}) based on the information $N_n$ will be of the form 
\begin{equation}\label{eq:alg-general}
A_n(f) \,=\, \phi_n\circ N_n(f) \,=\, \phi_n\bigl(\ell_1(f),\hdots,\ell_n(f)\bigr), \quad f\in H,
\end{equation}
where 
$\phi_n\colon\R^n\to G$ is an arbitrary mapping. Linear approximation is concerned with the case of linear $A_n\colon H\to G$. 
The \emph{worst-case error (w.c.e.)} of an algorithm $A_n$ as in \eqref{eq:alg-general} 
is then defined by
\[
e(A_n,F,G) \,:=\, \sup_{f\in F}\, \norm{f-A_n(f)}_G
\]
and any upper bound on $e(A_n,F,G)$ guarantees an a priori error bound on $A_n$ in the class $F$. 
Such a bound should be compared to the best possible error bounds for the given (class of) information. 

\goodbreak

First, if an information mapping $N_n$ as in~\eqref{eq:info} is fixed, 
define 
the \emph{radius of information} $N_n$ by 
\begin{equation}\label{eq:error-N}
r(N_n,F,G) \,:=\, 
\inf_{
\phi_n\colon \R^n\to G
}\, 
\sup_{f\in F}\, 
\Big\|f - \phi_n\circ N_n(f)\Big\|_{G}
\end{equation}
which quantifies the quality or \textit{power} of $N_n$. 
This should be seen relative to the ``best information'' from the class $\Lambda$ 
which gives rise to the \emph{$n$-th minimal error of information from $\Lambda$} 
defined by 
\vspace{-1mm}
\begin{equation*}
e_n(F,G,\Lambda) 
\,:=\, \inf_{N_n\in\Lambda^n}\, r(N_n,F,G)
\,=\, \inf_{A_n}\, e(A_n,F,G),
\end{equation*}
where the latter infimum is over all algorithms of the 
form~\eqref{eq:alg-general} with $\ell_1,\dots,\ell_n\in\Lambda$.

The number $e_n(F,G,\Lambda)$ (or rather the associated sequence) quantifies the power of optimal information and serves as the benchmark for any information obtainable from $\Lambda$. With this benchmark at our disposal, we clarify what ``iid'' information is.

\goodbreak 

Independently and identically distributed (iid) information is given by 
independent random continuous functionals $\ell_1,\dots,\ell_n\in\Lambda$, either on $F$ or $H$ with respect to a metric or norm, 
defined on some probability space $(\Omega,\Sigma,\mathbb{P})$ 
and having a common distribution $\nu$ on $\Lambda$. It is necessarily nonadaptive.
A random measurement of $f\in F$ is given by a random variable $\ell_j(f)$, 
i.e., 
a realization $\ell_j(f)^{\omega}$, $\omega\in \Omega$, of $\ell_j(f)$ 
is the application of the realization $\ell_j^{\omega}\in \Lambda$ of $\ell_j$ to $f$. 
Thus, $N_n^{(\nu)}(f)=(\ell_1(f),\dots,\ell_n(f))$ is a random vector 
with distribution $f_*(\nu)^{\otimes n}$ on $\R^n$, where $f_*(\nu)$ is the pushforward measure under $\ell\mapsto \ell(f)$.\\
(We suppress the $\omega$ and the corresponding probability space in the following.)

For each realization of $N_n^{(\nu)}$ we study again the minimal error 
that can be achieved with this information. 
That is, we consider the random variable 
\[
e_n^{iid}(F,G,\nu) \;:=\; r(N_n^{(\nu)},F,G),
\]
which we call the \emph{$n$-th minimal error of iid information} w.r.t.~$\nu$ 
from~$\Lambda$, 
see also~\eqref{eq:error-N}. It is clearly of interest to study characteristics of the above random variable, 
but it is still not clear what a reasonable quantity for this is. We use for example bounds holding in expectation or with high probability (whp), 
i.e., with probability tending to one as $n$ goes to infinity. In any case, we ignore events of measure zero.

The distribution or probability measure $\nu$ depends on the problem. 
To illustrate this, let us discuss some examples which 
are also 
our main applications. 

First, 
we consider 
\emph{standard information}, i.e.,  
if $\Lambda=\Lambda^\std$, 
for 
$F\subset L_2(\mu)$ with some 
measure $\mu$ on $D$. 
(Additional assumptions will guarantee that point evaluations are well-defined.) 
A distribution on this class of information corresponds
to a distribution $\nu$ on the domain $D$, 
if we consider 
$\ell(f)=f(X)$ 
where $X$ has distribution $\nu$ on $D$.
A natural choice of distribution is given by $\nu=\mu$, 
but we sometimes need another distribution for proving ``near-optimal'' results. 

On the class of arbitrary linear functionals we consider a Gaussian measure giving rise to \emph{Gaussian information}, i.e. Gaussian random functionals. Employing an interpretation from the finite-dimensional setting, 
this corresponds to the radius of the intersection of $F$ with a random subspace, 
a classical problem from geometric analysis and Banach space theory, see Section~\ref{sec:gauss-lin}. 

Finally, let us also mention \emph{random Fourier coefficients}, which are
given by $\ell(f)=\ipr{f, b_K}$ with 
a fixed orthonormal basis $\{b_k\}_{k\in\N}$ of $H$, 
where $K$ has distribution $\nu$ on $\N$ and $H$ is a Hilbert space. \\
We will present upper and lower bounds in all these cases.

\bigskip
\goodbreak

As it is our main object of study, 
we will use ``random information'' mostly synonymously 
with ``iid information'' and thus frequently speak of the power of 
random information. 
To summarize, we study $e_n^{iid}(F,G,\nu)$ for 
\begin{itemize}
\itemsep=1mm
	\item approximation of functions from a class $F\subset H$, 
	\item where the error is measured in a normed space $G$ with $F\subset G$, and 
	 \item with a \emph{random information mapping} $N_n^{(\nu)}\colon F\to\R^n$,
    \item where $\nu$ is a probability measure on the class of admissible information $\Lambda$.
\end{itemize}

\medskip
\goodbreak

We state several reasons for studying iid information in this setting:
\begin{enumerate}
	\item 
				If $e_n^{iid}(F,G,\nu)$ is ``small'' 
				with positive probability, we get an upper bound for $e_n(F,G,\Lambda)$, i.e., the error of \emph{optimal approximation}, without the need of finding a 
				sophisticated construction. This is in the spirit of the well-known \emph{probabilistic method}.
	\item  In cases where the \emph{minimal worst-case error} is known, one might wonder whether the optimal information is somehow special. One approach to this is to study $e_n^{iid}(F,G,\nu)$, and see whether 
				it is with high probability close to optimal, 
				showing that ``almost any'' information is good, 
				or not, which indicates that a more involved 
				construction is needed.
	\item It is a typical assumption in applications, 
				such as machine learning, that information (or data) is given by iid samples with respect to an unknown distribution.  It is therefore of interest to identify classes $F\subset G$ 
				and distributions on $\info$ 
				that allow for reliable error guarantees (such as whp for all $f\in F$) for random information. 
    \item Further, iid information is often \emph{universal}, that is, useful for many different problems and not particular to a certain problem instance or function class. This is beneficial if the available a priori knowledge is insufficient.  
    In contrast, deterministic constructions of such universal methods are often unknown.
\end{enumerate}

Obviously, the idea of employing randomness for worst-case analysis is not new. 
However, the power of iid information for linear approximation 
in a rather general setting, even for restricted classes of information,  
seems to have been observed only recently. To the best of our knowledge, it was only in the survey \cite{HKNPU20} that a systematic study was initiated. 
Since then, there has been some major progress in the case of Gaussian random information and random standard information. In particular, 
the power of random information has been determined precisely for certain 
natural choices of~$F$,~$G$ and~$\nu$, 
and also some general relations between minimal errors for different 
classes~$\info$ have been obtained. The present work aims to survey these recent results and to put them into a general framework. Thus, it can be understood as an update to the survey~\cite{HKNPU20}.

\medskip

Let us note that this work
builds on the PhD thesis \cite{Son22} of the first, and the habilitation thesis of the second author which both contain many of the mentioned recent results obtained together with several coauthors. 

\medskip

\subsection{What this survey is not about}

There are numerous aspects in (optimal) numerical approximation where randomness plays an important role. 
However, this survey is about aspects that are special 
to optimal, deterministic approximation based on iid information in the setting described above,
and we therefore refrain from discussing 
indirectly related subjects and results in detail. 

The most important omissions we are aware of consist of the following:
\begin{itemize}
\itemindent=-10pt
\item[] \textbf{Nonlinear algorithms.} Random constructions of ``good'' algorithms played a major role in numerical analysis in the last decades to tackle problems where explicit constructions are not available. 
Often, this is the case for problems where nonlinear algorithms are required, such as in compressed sensing. 
We show that sharp results, 
and some interesting open problems, 
can be found in the case of linear algorithms, too. 
\item[] \textbf{Randomized error criterion.} The setting considered in this survey should not be confused with the study of randomized algorithms (sometimes called Monte Carlo methods) with regard to probabilistic bounds on the error for \emph{each} individual $f$. 
    Here, although we assume that data is produced by random functionals, 
    this information is used for all $f\in F$ \emph{simultaneously} which in general is a stronger error criterion.
	\item[] \textbf{Other distributions.}
				In the following, we often assume that the information 
				is iid with respect to a given ``optimized'' or ``natural'' distribution. 
    It is clearly of interest, but not our focus, to study the effect of using other distributions or ``non-iid'' randomness.
        \item[] \textbf{Implementation and computational cost.} 
        We are not concerned with implementation cost of specific algorithms but
        focus on information complexity, which is in general only a lower bound on the total computational cost. 
        \item[] \textbf{Adaptive algorithms.} 
        In many cases it is interesting to study \emph{adaptive information/algorithms}. 
        We will only discuss non-adaptive algorithms 
        and note that, if $F$ is convex and symmetric, then the corresponding minimal errors (in the deterministic setting) differ by at most a factor of 2, see~\cite{GM80} and also~\cite[Section~4.2.1]{NoWo08}. 
\end{itemize}

The first three topics will be briefly discussed in Section~\ref{sec:further} together with additional aspects such as subsampling, learning and tractability of high-dimensional problems. We have to omit many areas where iid information has proven useful, both within and outside of the IBC-framework. These include, for example, density estimation, discretization, numerical integration (Monte Carlo) and Bayesian inference.

\paragraph{Notation:}
For a measure space $(D,\Sigma,\mu)$, 
we write $\mathcal{L}_p(\mu)$
for the set of $p$-integrable functions with usual norm, and inner product for $p=2$, 
and denote by $L_p:=L_p(\mu)$ 
the normed space of corresponding equivalence classes. Whenever convenient, we identify a function with its equivalence class. 
Moreover, we write $F\hookrightarrow G$ for two metric spaces 
$F\subset G$ (with possibly different metrics), 
and say that \emph{$F$ is embedded into $G$}, 
if the identity 
$\id\colon F\to G$, $\id(f)=f$, is a continuous 
injection. 
(If $G$ consists of equivalence classes, e.g., for $G=L_p$, then we use the usual modifications.)
For two sequences,  
$(e_n)_{n\ge0}$ and $(g_n)_{n\ge0}$, 
we write 
$e_n \lesssim g_n$ for $e_n\le C\, g_n$ for some constant $C>0$ 
and all $n\ge2$, and $e_n \asymp g_n$ if $e_n \lesssim g_n$ and $g_n \lesssim e_n$. 
If the sequences depend only on certain parameters, say $d$ or $s$, 
we write, e.g., $e_n \lesssim_{s,d} g_n$ and $e_n \asymp_{s} g_n$, respectively, 
to indicate the dependencies of the hidden constants. 
Without indication it may depend on all involved parameters, 
except for $n$. 
Given two, possibly infinite, square matrices $A$ and $B$ we indicate the Loewner order by $A\ge B$ meaning that $A-B$ is positive semi-definite. Similarly, we use $A\le B$ for $B\ge A$. The infinite identity matrix representing the identity $\id\colon \ell_2\to \ell_2$ is denoted by $I$ and the $n\times n$ identity matrix by $I_n$.

\section{Some benchmarks of optimal approximation} 
\label{sec:benchmarks}

In order to assess the power of random information for numerical approximation, optimal information will serve as a benchmark. Depending on the allowed algorithms or information, this gives rise to different concepts related to the \emph{widths} or \emph{$s$-numbers} of embeddings.  We refer to the monographs \cite{No88,Pi74,Pinkus85} for more information.

\paragraph{Approximation numbers} \hspace{-3mm}
are the minimal errors achievable by an arbitrary linear algorithm. 
That is, 
we define the 
\emph{$n$-th approximation numbers} 
(which are sometimes called \emph{linear $n$-widths}) 
of $F\subset H$ in $G$
by 
\begin{equation}\label{eq:an}
a_n(F,G) \,:=\, 
\inf_{\substack{\ell_1,\dots,\ell_n\in \Lambda^{\all}\\ g_1,\dots,g_n\in G}}\, 
\sup_{f\in F}\, 
\Big\|f - \sum_{i=1}^n \ell_i(f)\, g_i\Big\|_{G}.
\end{equation}
It is well known that nonlinear algorithms are often superior to linear ones, 
i.e., $e_n(F,G,\Lambda^{\all})<a_n(F,G)$ holds. 
However, equality holds, e.g., 
if $F=B_H$ is the unit ball of a Hilbert space $H$, 
or if $F$ is convex and symmetric and $G=L_\infty$.  
In such cases,  $a_n(F,G)=e_n(F,G,\Lambda^{\all})$, i.e., 
it is enough to consider linear algorithms,  
see e.g.~\cite[Section~4.2]{NoWo08}. 
Moreover, if $F$ is the unit ball $B_H$ in a Banach space $H$ and $H\hookrightarrow G$, then $a_n(F,G)\le (1+\sqrt{n})\,e_n(F,G,\Lambda^\all)$, see~\cite[Theorem~4.9]{NoWo08}. 
In addition, linear algorithms have 
(practical) advantages, 
which are not part of this survey.
Let us just note that the theory of linear approximations 
is much more developed than its nonlinear counterpart, 
with typical techniques such as linear regression, 
(polynomial) interpolation and 
projections on certain subspaces.

\paragraph{Kolmogorov widths} \hspace{-3mm}
are another prominent benchmark. 
The \emph{Kolmogorov $n$-width} of a set $F\subset G$ is defined by 
\begin{equation}\label{eq:dn}
d_n(F,G) \,:=\, 
\inf_{\substack{V_n\subset G\\ \dim(V_n)=n}}\, 
\sup_{f\in F}\; 
\inf_{g\in V_n}\, 
\Big\|f - g\Big\|_{G}, 
\end{equation}
i.e., it is the minimal distance (in $G$) that is achievable if we were to choose 
the \emph{best} element from a linear subspace 
of dimension $n$.

For this reason, 
the Kolmogorov widths are in general not related to the theory 
of algorithms: The inner infimum may not be attained by any (linear) algorithm and it appears to be 
an ``unfair'' benchmark.

Still, it is an essential tool in many arguments as it corresponds 
to the existence of good subspaces which can be used to define algorithms. There are several relations between $d_n$ and $a_n$, as well as to Gelfand widths via duality theory. 
For example, we have $d_n(F,L_2)=a_n(F,L_2)$ for $F\subset L_2$ since in this case the best approximation in a subspace is given by orthogonal projection.

\paragraph{Gelfand numbers} \hspace{-3mm}
are closely related to the minimal worst-case errors achievable with 
arbitrary algorithms based on arbitrary linear information.   
That is, if we define the \emph{$n$-th Gelfand number} 
of $F\subset H$ in $G$ by 
\begin{equation}\label{eq:cn}
c_n(F,G) \,:=\, 
\inf_{\substack{W_n\subset H\\ {\rm codim}(W_n)\le n}}\, 
\sup_{f\in F\cap W_n}\, 
\|f\|_{G},
\end{equation}
then it is well known that $c_n(F,G)$ differs from $e_n(F,G,\Lambda^{\all})$
by a factor of at most 2 whenever $F\subset H$ is convex and symmetric, 
see e.g.~\cite[Section~4.2]{NoWo08}. 
See also~\cite{CDD09,Dono06} for a more general version of this equivalence.

\goodbreak

\paragraph{Sampling numbers} \hspace{-3mm}
are the minimal worst-case errors 
that can be achieved with 
algorithms based on
function values as information, i.e., 
\begin{equation}\label{eq:gn}
g_n(F,G) \,:=\, e_n(F,G,\Lambda^{\std}) \,=\,
\inf_{\substack{x_1,\dots,x_n\in D\\ \phi\colon \R^n \to G\\}}\, 
\sup_{f\in F}\, 
\Big\|f - \phi\!\left(f(x_1),\dotsc,f(x_n)\right) \Big\|_{G}, 
\end{equation}
where $F\subset G$ are classes of functions on the set $D$.

The ubiquity of function values in applications might suggest that sampling numbers have been studied in depth as a benchmark.
However, 
besides plenty of results in specific settings, 
there are only few general results about them.
An example is a recent optimal bound for $L_2$-approximation in Hilbert spaces, which is based on the general results on iid information and subsampling described in Section~\ref{subsec:sub}, 
see Theorem~\ref{thm:sub-general}.

Restricting to linear algorithms gives rise to the 
\emph{linear sampling numbers} 
\begin{equation*}
g_n^{\lin}(F,G) \,:=\, 
\inf_{\substack{x_1,\dots,x_n\in D\\ g_1,\dots,g_n\in G}}\, 
\sup_{f\in F}\, 
\Big\|f - \sum_{i=1}^n f(x_i)\, g_i\Big\|_{G},
\end{equation*}
which can be used to bound $g_n(F,G)$ from above.

\begin{rem}
We comment on the difference of ``width'' and ``numbers'' in the above context. 
While the width of a \emph{set} $F\subset G$ 
can be defined solely based on the knowledge of $F$ and 
(the norm of) $G$, the definition of the (approximation/Gelfand) numbers also requires the 
normed space $H$ containing $F$ to define the class $\Lambda^{\all}$ of all linear functionals. 
We refer again to \cite{Pi74,Pinkus85}, and note that there is also a concept of 
Gelfand width that can yield different results, see \cite{EL13}. This can also be the case for linear widths/approximation numbers, see  \cite{Hei89}, and Remark 2.8 in \cite{JUV23}.
\end{rem}

\goodbreak 

\section{Approximation based on iid information}
\label{sec:general}

In the following we develop a general approach for approximation based on random information 
and a very simple linear algorithm: A \emph{weighted least squares method} 
on a suitable subspace and with suitable (explicitly given) weights. \\
This will be done in four steps: We show
\begin{enumerate}
	\item an $L_2$-error bound for given information mappings, 
	\item how this 
				relates to (infinite) matrices in the case of Hilbert spaces, 
	\item how concentration inequalities lead to optimal bounds, and finally, 
	\item how we can treat more general classes and 
				approximation w.r.t.~other norms.
\end{enumerate}

\subsection{$L_2$-approximation and least squares methods}

We first treat the case of approximation in $L_2$-norm. For this, we fix some measure space $(D,\Sigma,\mu)$, 
and write $L_p:=L_p(\mu)$. 
We consider the \emph{weighted least squares estimator}
\begin{equation}\label{eq:alg}
A_{N}(f) \,:=\, 
\underset{g\in V_n}{\rm argmin}\, \sum_{i=1}^N 
w_i\, \vert \ell_i(f-g) \vert^2
\end{equation}
for some subspace $V_n$ of $L_2$ of dimension $n$, 
some weights $w_i>0$ and linear functionals 
$\ell_i\in\Lambda$, $i=1,\dots,N$. Note that this map is well-defined and linear under condition \eqref{eq:prop-cond} below, see Proposition~\ref{prop:LS} and its proof. 
Since we have different types of admissible information $\Lambda$ in mind,
we present results in more general form than in the literature.

\medskip

In any case, the algorithm $A_N$ in~\eqref{eq:alg} is well studied, 
see e.g.~the recent contributions~\cite{Bo17,CCMNT15,CDL13} 
and the references therein.
It seems surprising that,  when fed with random information, this simple method often leads to (near-)optimal bounds. 
The next result 
forms a basis for many results 
discussed in this survey, 
such as the ones in~\cite{DKU23,HKNPU21,KU21}. 

\begin{prop}
\label{prop:LS}
Let $H\subset L_2$ be a normed space of functions on $D$, 
and, 
for $n\le N$, 
let $V_n$ be a $n$-dimensional subspace of $H$, and
$\ell_1,\dots,\ell_N$ be linear functionals on 
$H$.
Assume that 
\begin{equation}\label{eq:prop-cond}
\inf_{g\in V_n} \frac{\sqrt{\sum_{i=1}^N w_i\,|\ell_i(g)|^2}}{\|g\|_{L_2}} 
\;\ge\; \alpha 
\end{equation}
for some 
$\alpha>0$ 
and some weights $w_1,\dots,w_N>0$.
Then, for all $f\in H$ and $g\in V_n$, the algorithm from~\eqref{eq:alg} with 
the corresponding $V_n$, $\ell_i$ and $w_i$ is well-defined and satisfies 
\begin{equation*}
\Big\|f - A_N(f)\Big\|_{L_2}
\;\le\; \big\Vert f - g\big\Vert_{L_2} 
\,+\, \frac{1}{\alpha} \,
 \sqrt{\sum_{i=1}^N w_i\,\abs{\ell_i(f-g)}^2}.
\end{equation*}
\end{prop}

\medskip

Note that, for studying the error over a set $F\subset H$, 
it is necessary to assume linearity of the used functionals on the affine spaces $f+V_n$ for each $f\in F\cup\{0\}$, 
and introducing such a surrounding normed space $H$ seems a convenient way to do so.

\bigskip

The above proposition shows that information functionals $\ell_i$ are ``good'' 
for $L_2$-approximation in $F$ if one can choose a subspace $V_n$ and 
weights $w_i$ such that 
the squared sum in \eqref{eq:prop-cond}
has large values on $V_n$ 
and, 
for each $f\in F$, there is some $g\in V_n$ such 
that $\|f-g\|_{L_2}$ and $\sum_{i=1}^N w_i\,\abs{\ell_i(f-g)}^2$ 
are small.

\medskip

Condition~\eqref{eq:prop-cond} 
says that the discrete (semi)norm based on the functionals and given weights should be comparable to the $L_2$-norm on $V_n$. 
Finding such functionals and weights in the case of standard information is called \textit{discretization} 
(another interesting topic that we do not discuss in detail). 
That is, one wants to find a point set such that the $L_2$-norm of all functions from some $n$-dimensional $V_n$ can be ``discretized'' using function values at these points. 
Also here, random points are used as a tool, and variants of Proposition~\ref{prop:LS} appear. 
We refer to~\cite{DT22,Groe20,KKLT22} and references therein.

\medskip

For approximation in classes of functions, we note that Proposition~\ref{prop:LS}, in general, also requires 
a suitable discretization of the ``remainder'' $f-g$, 
which appears to be more involved. 
This simplifies slightly 
if we compare the $L_2$-error of $A_N$ with 
best approximation in the uniform norm, see e.g.~\cite[Theorem~2.1]{Te20}. 
As usual, set $\|f\|_\infty:=\sup_{x\in D}|f(x)|$ for $f\in B(D)$, i.e., the space of bounded functions on~$D$.

\begin{coro}\label{cor:LS-infty}
Let $\mu$ be a finite measure and 
 $H\subset B(D)$ be a normed space of functions on $D$, 
and $\ell_1,\dots,\ell_N$ be bounded linear functionals on $B(D)$. 
Then, for each $V_n\subset H$ with \eqref{eq:prop-cond} for $w_i=\frac1N$ the corresponding unweighted least squares method from~\eqref{eq:alg} satisfies
\[
\qquad
\Big\|f - A_N(f)\Big\|_{L_2}
\;\le\; \frac{c}{\alpha}\, 
\inf_{g\in V_n} \|f-g\|_\infty
\qquad\text{ for all } f\in H 
\] 
with 
$c\le \sup\{\|f\|_{L_2}+|\ell_i(f)|\colon\, f\in H,\, \|f\|_\infty=1,\, i=1,\dots,N\}$.
\end{coro}

\medskip

In particular, if 
we are free to choose $V_n$, then we might want to take one that minimizes 
the right hand side. 
For finding corresponding functionals in the case of function evaluations, 
it has been observed in~\cite[Theorem~6.2]{BSU} 
that for every $V_n\subset C(D)$ of dimension $n$, 
there exist $x_1,\dots,x_N$ with $N\asymp n$ 
with \eqref{eq:prop-cond} for $w_i=1/N$; 
see also~\cite{CM17,LT22,Te20} for earlier results. 
Applied to the (near-)optimal subspace Corollary~\ref{cor:LS-infty} gives
\[
\sup_{f\in F}  
\Big\|f - A_N(f)\Big\|_{L_2}
\;\lesssim\; d_n(F,L_\infty)
\] 
with $A_N(f)=\argmin_{g\in V_n}\sum_{i=1}^N|f(x_i)-g(x_i)|^2$ with the Kolmogorov width as in \eqref{eq:dn}, see~\cite[Coro.~6.4]{BSU} and \cite[Thm. 1.1]{Te20}. 
This is proven using slightly larger iid point sets which are subsampled to obtain point sets of optimal size. 
This technique will be discussed briefly in Section~\ref{subsec:sub}. 
Following the same arguments, similar results can be obtained for more general classes of functionals.

\bigskip

For obtaining optimality of the used method $A_N$ among all linear algorithms, 
it is of interest to replace the $L_\infty$ on the right hand side 
by $L_2$ since $d_n(F,L_2)=a_n(F,L_2)$. 
This is possible under further (necessary) assumptions on $F$ 
and $\ell_i$, 
and will be the subject of the following sections.

\bigskip

We end this section with the proof of Proposition~\ref{prop:LS}, 
which follows very closely the lines of Section~3 in~\cite{KU21a}.

\begin{proof}[Proof of Proposition~\ref{prop:LS}] 
Let $V_n={\rm span}\{b_1,\dots,b_n\}\subset H$ for some orthonormal system in $L_2$. 
Then, the algorithm from~\eqref{eq:alg}
can be written as 
\[
A_{N}(f) \,=\, \sum_{k=1}^n \left(G^+ N(f)\right)_k\, b_k, 
\]
where $N\colon F \to \R^N$ with 
$N(f):=\left(\sqrt{w_i}\, \ell_i(f)\right)_{i\leq N}$ 
is the (weighted) \emph{information mapping} and
$G^+\in \R^{n\times N}$ is the Moore-Penrose inverse of the matrix 
\begin{equation*}
 G := \left(\sqrt{w_i}\,\ell_i(b_k)\right)_{i\le N,\, k\le n} 
\in \R^{N\times n},
\end{equation*}
whenever $G$ has full rank. \\
This is fulfilled, because \eqref{eq:prop-cond} is equivalent 
to $s_{n}(G)\ge\alpha$, where $s_{n}(G)$ denotes the $n$-th 
singular value of $G$. 
In particular, we have $A_N(g)=g$ for every $g\in V_n$ and $\norm{G^+}_{2\to2}=s_n(G)^{-1}\le\frac1\alpha$. Therefore,
\[\begin{split}
\norm{f-A_{N}(f)}_{L_2} 
\,&\le\,  \norm{f-g}_{L_2} + \norm{g - A_{N}(f)}_{L_2} 
\,=\,  \norm{f-g}_{L_2} + \norm{A_{N}(f-g)}_{L_2} \\
\,&=\, \norm{f-g}_{L_2}  + \norm{G^+ N (f- g)}_{\ell_2^n} \\
\,&\le\, \norm{f-g}_{L_2}  + \frac1\alpha\,\norm{N (f- g)}_{\ell_2^N}, 
\end{split}\]
which 
proves the claim.
\end{proof}

\bigskip
\goodbreak

\subsection{Hilbert spaces and (random) matrices}\label{sec:hilbert}

The very general result in Proposition~\ref{prop:LS} has useful implications for Hilbert spaces. In the following, we consider a separable Hilbert space $H$ which is continuously embedded into $L_2=L_2(\mu)$, where $\mu$ is a measure on 
some set.  
We shall assume that the norm of $H$ is given by
\begin{equation}\label{eq:H-norm}
\|f\|_H^2 
\;:=\; \sum_{k=1}^\infty \abs{\ipr{f,  \sigma_k b_k}_{H}}^2
\;=\; \sum_{k=1}^\infty \frac{|\ipr{f, b_k}_{L_2}|^2}{\sigma_k^2}, 
\end{equation}
where $\{b_1,b_2,\dots\}$ is an orthogonal basis of $H$ 
that is orthonormal in $L_2(\mu)$, and $\sigma_1\ge\sigma_2\ge\dots>0$. Note that by the spectral theorem 
such a basis exists and $\sigma_n\to 0$ holds, 
if the embedding $H\hookrightarrow L_2(\mu)$ is compact, see e.g.~\cite[Section~4.2.3]{NoWo08}. Further, it holds that $d_k(B_H,L_2)=a_k(B_H,L_2)=c_k(B_H,L_2)=\sigma_{k+1}$ for all $k\in\N$.  

\medskip

If we choose $V_n={\rm span}\{b_1,\dots,b_n\}$ as the \textit{optimal subspace} of $H$, 
see e.g.~\cite[Thm. IV.2.2.]{Pinkus85}, 
we obtain the following consequence of Proposition~\ref{prop:LS}, 
see e.g.~\cite{HKNPU21}.

\begin{prop}
\label{prop:LS-H}
Let $H\hookrightarrow L_2$ be a 
separable 
Hilbert space 
with 
norm as in \eqref{eq:H-norm}.
Moreover, 
for $N\ge n$, 
let 
$\ell_1,\dots,\ell_N$ 
be continuous 
linear functionals on~$H$, 
$P_n$ be the orthogonal projection onto 
$V_n:={\rm span}\{b_1,\dots,b_n\}$,  
and assume that 
\begin{equation}\label{eq:matrix-1}
\bigg(\sum_{i=1}^N w_i\, \ell_i(b_k) \,\overline{\ell_i(b_j)} \bigg)_{k,j=1}^n
\,\ge\, \alpha^2\,I_n
\end{equation}
and
\begin{equation}\label{eq:matrix-2}
\bigg(\sum_{i=1}^N w_i\, \ell_i(\sigma_k b_k)\, \overline{\ell_i(\sigma_j b_j)} \bigg)_{k,j=n+1}^\infty
\,\leq\, \beta^2\,I,
\end{equation}
for some 
$\alpha>0$, $\beta\ge 0$ and weights $w_1,\dots,w_N\ge0$.
Then, for all $f\in H$, 
the algorithm from~\eqref{eq:alg} with 
the corresponding $V_n$, $\ell_i$ and $w_i$ satisfies 
\begin{equation}\label{eq:coro-local}
\Big\|f - A_N(f)\Big\|_{L_2}
\;\le\; \left(\sigma_{n+1}+\frac{\beta}{\alpha}\right)
 \big\Vert f - P_n f\big\Vert_{H}.
\end{equation}
\end{prop}

\bigskip
\goodbreak

For the worst-case error over the unit ball $B_H$ of $H$ we obtain from \eqref{eq:coro-local} that 
\begin{equation} \label{eq:wce-H}
e(A_N,B_H,L_2)\le \sigma_{n+1}+\frac{\beta}{\alpha}. 
\end{equation}
Since we always have $e(A_N,B_H,L_2)\ge \sigma_{N+1}$, 
this leads to an optimal bound 
(up to constants) provided that 
$\frac{\beta}{\alpha}\lesssim \sigma_{n+1}$ with 
$N\asymp n$. 
Later, we will see that this can be obtained 
whp
for certain classes of random information.  
But first we consider a general result that works 
for many more classes of information if we allow a 
\textit{logarithmic oversampling}.

Let us also mention that it is well known that in the setting of Proposition~\ref{prop:LS-H} \emph{optimal information} $N_n^*$
is given by $\ell_i(\cdot)=\langle b_i,\cdot\rangle$, $i=1,\dots,n$. 
In this case, we can choose $N=n$ and equal weights $w_i=1$ 
to have $\alpha=1$ and $\beta=0$. 
Then, in fact, $A_N=P_n$ and equality holds in \eqref{eq:wce-H}, i.e.,
$
e(N_n^*,B_H,L_2)= \sigma_{n+1}.
$
This 
recovers the best possible bound for approximation in Hilbert spaces, see e.g. \cite[Thm. 4.11]{NoWo08}.

\bigskip

The above results show that the least squares algorithm 
from~\eqref{eq:alg} satifies ``good'' error bounds for all $f\in H$ at once, 
if weights and functionals can be found that satisfy 
\eqref{eq:matrix-1} and \eqref{eq:matrix-2} with large $\alpha$ and small $\beta$. For later use we restate these conditions in the form
\begin{equation}\label{eq:sing-1}
s_n\Big(\big(\sqrt{w_i}\, \ell_i(b_j) \big)_{1\le i \le N,1\le j\le n}\Big)
\,\ge\, \alpha
\end{equation}
and
\begin{equation}\label{eq:sing-2}
s_1\Big(\big(\sqrt{w_i}\, \ell_i(\sigma_j b_j) \big)_{1\le i \le N,j>n}\Big)
\,\leq\, \beta,
\end{equation}
where $s_k(A)$ denotes the $k$-th singular value of a matrix $A$.
In the following, we will see that the existence of good (iid) information 
can be guaranteed for rather general classes of information $\Lambda$ 
and Hilbert spaces~$H$. This is based on 
results from \emph{random matrix theory}, in particular the concentration result for sums of rank-one matrices in Lemma~\ref{lemma:concentration} below, which we shall apply to get \eqref{eq:matrix-1} and \eqref{eq:matrix-2} whp. 
In this way we obtain the ``local version'' of the main result 
from~\cite{KU21}, but in a slightly more general form, see also~\cite{U20}.

\bigskip

\begin{thm}
\label{thm:general-H} 
Let $H\hookrightarrow L_2$ be a 
Hilbert space 
with norm as in \eqref{eq:H-norm} 
and 
$\sum_{k}\sigma_k^2<\infty$. 
Moreover, let 
$\Lambda\subset H'$ 
be a class of information, 
and $\nu$ be a (possibly infinite) measure
on~$\Lambda$ with
\begin{equation}\label{eq:info-ass}
\int_\Lambda \ell(f)\cdot 
\ell(g) 
\,\dint\nu(\ell)
\;=\; \l f, g\r_{L_2} \quad \text{for all }f,g\in H.
\end{equation}
Further, let $\ell_1,\dots,\ell_N\in\Lambda$ be iid random functionals distributed with $\nu$-density
\begin{equation}\label{eq:density}
\rho_n(\ell) \;:=\; \frac{1}{2}\left(\frac{1}{n}\sum_{k\le n} |\ell(b_k)|^2
	+\frac{\sum_{k>n}\sigma_k^2\,|\ell(b_k)|^2}{\sum_{k>n}\sigma_k^2}\right), \quad \ell\in \Lambda.
\end{equation}
Then, 
with probability (at least) $1-N^{-c}$, 
the algorithm $A_N$ from~\eqref{eq:alg} with 
$V_n={\rm span}\{b_1,\dots,b_n\}$, 
$N\gtrsim_c n \log n$, 
functionals $\ell_i$
and weights $w_i=\rho_n(\ell_i)^{-1}$  
satisfies 
\begin{equation}\label{eq:local}
\Big\|f - A_N(f)\Big\|_{L_2}
\;\le\;3\max\bigg\{\sigma_{n+1}, \, \sqrt{\frac{1}{n}\sum_{k>n}\sigma_k^2}\bigg\} \; 
 \big\Vert f - P_n f\big\Vert_{H}
\end{equation}
for all $f\in H$, 
where
$P_n$ is the orthogonal projection onto~$V_n$. 
In particular,
\[
e_N^{iid}(B_H,L_2,\rho_{32n}\cdot\dint\nu) 
\;\le\; 
\sqrt{\frac1n\sum_{k> n} \sigma_k^2}
\]
with the same probability
if $N\gtrsim_c n \log n$. 
\end{thm}

\bigskip

The expression $\frac{1}{n}\sum_{k\le n} |\ell(b_k)|^2$ in the density~\eqref{eq:density} 
is also called (the inverse of) the \emph{Christoffel function}, 
at least for 
$\ell(f)=\delta_x(f)$ being given by function evaluations, 
see e.g.~\cite{CM17} and the references therein. 

Another way to write the density is
\[
\rho_n(\ell) \;\sim\; 
 \sup_{g\in V_n}\frac{|\ell(g)|}{\|g\|_{L_2}}
\,+\,	\tau_n\cdot\sup_{g\in H\cap V_n^\perp}\frac{|\ell(g)|}{\|g\|_{H}}	
\]
with 
$\tau_n:=(\frac1n\sum_{k> n} \sigma_k^2)^{-1}$.

\medskip

To illustrate the effect of the sampling density $\rho_n$, 
which also appears inversely in the weights in the algorithm $A_N$, 
assume that the $\ell(b_k)$ can attain arbitrarily large and small values. 
Roughly speaking, while functionals with large values of $\rho_n(\ell)$ are likely to be sampled, 
their contribution in the algorithm is ``damped'' and vice versa for functionals for which the density small. It is not entirely clear why these competing effects are beneficial. 

We recently learned that  a density similar to~\eqref{eq:density} has been used in~\cite{FB17} 
in the context of \emph{random feature expansions}. 
Hence, this density appears useful also in the context of constructing ``good'' random subspaces for approximation, 
in contrast to our emphasis on finding ``good'' information.

Let us also add that sampling according to $\rho_n$ might be non-trivial. 
However, in many cases, like 
for Gaussian information or 
for function evaluations for certain Sobolev spaces 
one can sample instead from the underlying measure $\nu$.

\medskip

We end this section by providing the proofs of Proposition~\ref{prop:LS-H} and Theorem~\ref{thm:general-H}.

\begin{proof}[Proof of Proposition~\ref{prop:LS-H}]
Since 
\[
\bigg(\sum_{i=1}^N w_i\, \ell_i(b_k) \,\overline{\ell_i(b_j)} \bigg)_{k,j=1}^n
\,\ge\, \alpha^2\,I_n
\]
is equivalent to~\eqref{eq:prop-cond}, i.e.,
\begin{equation*}
\inf_{g\in V_n} \frac{\sqrt{\sum_{i=1}^N w_i\,|\ell_i(g)|^2}}{\|g\|_{L_2}} 
\;\ge\; \alpha,
\end{equation*}
we obtain from Proposition~\ref{prop:LS} with $g:=P_nf$ that
\[\begin{split}
\Big\|f - A_N(f)\Big\|_{L_2}
\;&\le\; \big\Vert f - P_nf\big\Vert_{L_2} 
\,+\, \frac{1}{\alpha} \,
 \sqrt{\sum_{i=1}^N w_i\,\abs{\ell_i(f-P_nf)}^2} \\
\;&\le\; \sigma_{n+1}\,\norm{f-P_n f}_{H} +\frac{1}{\alpha} 
\sqrt{\sum_{i=1}^N w_i\,\abs{\ell_i(f-P_nf)}^2}.
\end{split}\]

Note that for every $f\in H$ we have $f-P_nf\in H\cap V_n^\perp$ and, by continuity, $\ell_i(f-P_nf) =\sum_{k>n}\ipr{f,\sigma_k b_k}_H\ell_i(\sigma_k b_k)$, 
$i=1,\dots,N$. From 
\[
\bigg(\sum_{i=1}^N w_i\, \ell_i(\sigma_k b_k)\, \overline{\ell_i(\sigma_j b_j)} \bigg)_{k,j=n+1}^\infty
\,\leq\, \beta^2\,I,
\]
which is equivalent to
\begin{equation} \label{eq:cond2}
\sup_{g\in H\cap V_n^\perp} \frac{\sqrt{\sum_{i=1}^N w_i\,|\ell_i(g)|^2}}{\|g\|_{H}} 
\;\le\; \beta,
\end{equation}
we obtain
\[
\sqrt{\sum_{i=1}^N w_i\,\abs{\ell_i(f-P_nf)}^2}
\;\le\; \beta\, \|f-P_n f\|_H
\]
which completes the proof.\\
\end{proof}

\bigskip

The proof 
of Theorem~\ref{thm:general-H} 
is based on 
the following concentration inequality for infinite matrices, 
which was essentially proved 
in~\cite[Theorem 2.1]{MP06}, 
see also~\cite{AW02,O10,Tr12}. 
Here we use a tailored reformulation of this result from
\cite[Prop.~1]{U20}.

\begin{lemma}\label{lemma:concentration}
Let $N\geq3$, $R>0$ and $y_1,\dots,y_N$ be i.i.d.\ random sequences 
from $\ell_2(\N)$ 
satisfying $\|y_i\|_{2}^2\leq R^2$ almost surely and $\|E\|_{2\to 2}\leq1$, 
where $E=\E(y_i y_i^*)$.
	
Then
\[
\PP\left(\bigg\|\frac{1}{N}\sum_{i=1}^N y_i y_i^* - E\bigg\|
	\,>\, \frac12\right)
\,\le\, \frac{4}{N^c},
\]
whenever 
\begin{equation*}
\frac{N}{\ln(N)} \;\ge\; 64(2+c)\,R^2,
\end{equation*}
e.g., if 
$N\ge 2^9(1+c)\ln(e+c) R^2\ln(e R^2)$ for $R\ge2$.

\end{lemma}

(Numbers in the last line have been checked numerically.)

\bigskip

\begin{proof}[Proof of Theorem~\ref{thm:general-H}]

We follow the lines of~\cite{KU21} (see also~\cite{DKU23}).

In order to apply Proposition~\ref{prop:LS-H}, it remains to show that~\eqref{eq:matrix-1} and~\eqref{eq:matrix-2} hold for some $\alpha,\beta>0$ with sufficiently high probability. 
Define the random vectors
\[
(y_i)_k=\left\{\begin{array}{cl} \sqrt{w_i}\, \ell_i(b_k),&\quad \text{ if }1 \le  k\le n,\\[2mm] 
\displaystyle\frac{\sqrt{w_i}\, \ell_i(\sigma_k\,b_{k})}{\gamma_n},&\quad \text{ if } 
k>n,
\end{array}\right.
\]
where
\[
\gamma_n \,:=\, 
\max\bigg\{\sigma_{n+1}, \, \sqrt{\frac{1}{n}\sum_{k>n}\sigma_k^2} \, \bigg\}
>0.
\]
By the choice of the weights $w_i=\rho_n(\ell_i)^{-1}$, we obtain 
\[
\|y_i\|_2^2 \;=\; w_i \left(\sum_{k\le n} |\ell_i(b_k)|^2
	+ \gamma_n^{-2}\sum_{k>n}\sigma_k^2\,|\ell_i(b_k)|^2 \right) 
\;\le\; 2n.
\]
Further,
\[
(y_i y_i^*)_{k,j} \;=\;
\begin{cases} 
w_i\, \ell_i(b_k) \,\overline{\ell_i(b_j)},&\quad \text{ if }1 \le  k,j\le n,\\[2mm] 
\displaystyle\frac{w_i\, \ell_i(\sigma_k\,b_{k}) \,\overline{\ell_i(\sigma_j\,b_{j})}}{\gamma_n^2},&\quad \text{ if } 
k>n,
\end{cases}
\]
and using~\eqref{eq:info-ass} we obtain  
\[
\E y_i y_i^* 
\;=\; \begin{pmatrix} I_n &0\\0&C\end{pmatrix},
\]
where $C$ is an infinite diagonal matrix with spectral norm at most one, 
and the expectation is with respect to $\rho_n\cdot \dint\nu$. 

Lemma~\ref{lemma:concentration} yields  
\[
\bigg(\sum_{i=1}^N w_i\, \ell_i(b_k) \,\overline{\ell_i(b_j)} \bigg)_{k,j=1}^n
\,\ge\, N I_n-\frac{N}2\,I_n
\,=\, \frac{N}2\,I_n
\]
and
\[
\bigg(\sum_{i=1}^N w_i\, \ell_i(\sigma_k b_k)\, \overline{\ell_i(\sigma_j b_j)} \bigg)_{k,j=n+1}^\infty
\,\leq\, N\gamma_n^2(C \;+\; \frac12 I)
\,\leq\, \frac{3N}2 \gamma_n^2\, I,
\] 
simultaneously with probability $1-4N^{-c}$, if $N \ge C' n \log n$ 
for some constant $C'>0$ that only depends on $c$.
Hence, we have that the algorithm $A_N$ from Proposition~\ref{prop:LS-H} 
satisfies, with the same probability, that 
\[
\Big\|f - A_N(f)\Big\|_{L_2}
\;\le\; \left(\sigma_{n+1}+3\gamma_n\right)
 \big\Vert f - P_n f\big\Vert_{H}
\]
for all $f\in H$. 
Using that
\[
\gamma_n^2 =\max\bigg\{\sigma_{n+1}^2, \, \frac{1}{n}\sum_{k>n}\sigma_k^2 \, \bigg\}
\,\le\,\frac{2}{n}\sum_{k> n/2}\sigma_k^2,
\]
one obtains the bound 
\[
\Big\|f - A_N(f)\Big\|_{L_2}
\;\le\;\sqrt{\frac{32}{n}\sum_{k\geq\lceil n/2\rceil}\sigma_k^2} \; 
 \big\Vert f - P_n f\big\Vert_{H}.
\]
Replacing $n$ by $32n$ and taking the supremum over $f\in B_H$ implies the uniform bound over $B_H$.
\end{proof}

\bigskip

\subsection{General classes of functions}

The purpose of this and the following section is to transfer the above results for $L_2$-approximation in Hilbert spaces
to more general situations. 
In fact, 
it will turn out that 
the ``local'' result from the last section 
(referring to the expression $\|f-P_nf\|_H$ on the right hand side 
of \eqref{eq:local}) 
can be used directly to prove bounds for more general classes of functions. 

The main idea is to construct a Hilbert space $H$ containing $F$, 
and then apply 
Theorem~\ref{thm:general-H} 
to $H$. 
In the context of random information, this approach was used already in~\cite{KU21a}. 
Another construction of such an $H$, which is also the one we employ here, 
has been found recently in~\cite{KPUU23}.

Unfortunately, we need that the information functionals~$\ell_i$ 
are continuous on this Hilbert space, which seems difficult to formalize in 
a nice form. In fact, as the last proof showed, we need that 
$\ell(f)=\sum_{k\in\N}\ipr{f,b_k}_{L_2} \ell(b_k)$ holds for all $f\in F$ (or a dense subset) for almost all $\ell\in\Lambda$ with respect to the chosen measure on $\Lambda$. 
Here, we work under the following assumption. 

\bigskip

\hypertarget{assumA}{}
\noindent\textbf{Assumption A. } 
We consider the following setting. 
\begin{itemize}
    \item[(A.1)] $F$ is a separable metric space of functions on a set $D$. 
    \item[(A.2)] $\mu$ is a measure on $D$ such that $F$ is continuously embedded into $L_2:=L_2(\mu)$. \\
    $\Lambda$ is a class of functionals on $\mathcal{L}_2(\mu)$,   
    the $\mu$-square integrable functions on~$D$. 
    
   \item[(A.3)] 
   $\nu$ is a measure on~$\Lambda$ such that 
   $\nu$-almost every ($\nu$-a.e.) 
   $\ell\in \Lambda$ is continuous on $F$ and  
  for every $f\in L_2$ the map $\ell\mapsto \ell(f)$ belongs to $L_2(\nu)$, in particular is well-defined, and
\begin{equation*}
\int_\Lambda \ell(f)\cdot 
\ell(g) 
\,\dint\nu(\ell)
\;=\; \l f, g\r_{L_2},\qquad f,g\in L_2. 
\end{equation*}
\end{itemize}

\bigskip

Assumption (A.3) seems to be quite restrictive, but it fits the examples 
we have in mind, like function evaluations or Gaussian information, see Section~\ref{sec:improve}.  However, other classes of information, like derivative values or \emph{local averages} of the input, see e.g.~\cite{Boy05}, do not satisfy this assumption, meaning that we cannot find a measure $\nu$ such that (A.3) holds.   
Note that (A.3) readily implies that, for all $f,g\in L_2$, we have $\ell(f+g)=\ell(f)+\ell(g)$ for $\nu$-a.e.~$\ell\in\Lambda$, 
see the proof of Lemma~\ref{lemma:emb}.

\bigskip 

\begin{rem}
Note that in IBC, or numerical analysis in general, a common assumption is that the information functionals must be defined 
solely on the class $F$, or a ``surrounding'' normed space 
$H\supset F$. 
Our analysis, however, requires that we can extend a.e.~$\ell\in\Lambda$ 
also to more general subspaces of $L_2$. 
This led us to \hyperlink{assumA}{Assumption~(A.3)},  
which may be weakened to equality up to constants independent of $f,g\in L_2$.   
It would be interesting 
to find for given~$F$ necessary conditions on $\Lambda$ such that (A.3) holds. 
\end{rem}

\bigskip 

Under the above assumption we obtain the following general statement.

\begin{thm}
\label{thm:general}
There are constants $b,C\in\N$ such that for any 
$F$, $D$, $\mu$, $\Lambda$, $\nu$ that satisfy \hyperlink{assumA}{Assumption~A}
and all $n\in \N$
and $N\ge C n\log(n+1)$, we have
 \[
e_N^{iid}(F,L_2,\rho_{bn}\cdot\dint\nu)  \,\le\, 
 \frac{1}{\sqrt{n}} \sum_{k\ge n} \frac{d_k(F,L_2)}{\sqrt{k}}
 \]
with probability $1-N^{-42}$ where the density $\rho_{n}$ 
is given by~\eqref{eq:density}, 
with suitable~$\{b_k\}$ and 
$\sigma_k:=k^{-1/2} d_{\lfloor k/8\rfloor}(F,L_2)$. 
The bound is achieved by the corresponding algorithm~$A_N$ from 
Theorem~\ref{thm:general-H}.
\end{thm}

\bigskip

Recall that $d_k(F,L_2)=a_k(F,L_2)$ 
if $F\subset L_2$, and that $d_n(F,L_2)<\eps$ is equivalent to the existence 
of an $n$-dimensional $V_n\subset F$ with 
$\sup_{f\in F}\inf_{g\in V_n}\|f-g\|_{L_2}<\eps$.

\bigskip

The following technical lemma is a composition from 
Section~6.1 of~\cite{DKU23} and 
the proof 
of~\cite[Prop.~11]{KPUU23}, extended to general information.

\begin{lemma} \label{lemma:emb}
Let \hyperlink{assumA}{Assumption~A} be fulfilled and  
assume $d_k(F,L_2)>0, k\in \N$, with 
\begin{equation*} 
\sum_{k\in \N} \frac{d_k(F,L_2)}{\sqrt{k}} \;<\; \infty. 
\end{equation*}
Then, there is an ordered orthonormal system $(b_k)_{k\in \N}$ in $L_2$ 
such that $F\subset H\subset L_2$, where the Hilbert space $H$ is the completion of $\textrm{span}\, \{b_k\}_{k\in \N}$
with respect to the norm 
\[
\|f\|_H^2 
\;:=\; \sum_{k=1}^\infty 
		\frac{\sqrt{k}\,|\ipr{f, b_k}_{L_2}|^2}{d_{\lfloor k/8\rfloor}(F,L_2)} 
\]
and, for each 
$n\in8\N$, we have 
\begin{equation}\label{eq:F-H-bound}
\sup_{f\in F}	\Vert f - P_{n} f \Vert_H 
	\;\le\; 4\, \sqrt{\sum_{k> \lfloor n/8\rfloor} \frac{d_k(F,L_2)}{\sqrt{k}}},
\end{equation}
where 
$P_n$ is the orthogonal projection onto 
$V_n:={\rm span}\{b_1,\dots,b_n\}$. \\
Moreover, for every countable set $F_0\subset F$, 
there is some $\Lambda_0\subset\Lambda$ with $\nu(\Lambda\setminus\Lambda_0)=0$ such that 
\begin{equation}\label{eq:ell-cont} 
\ell(f) \;=\; \sum_{k=1}^\infty \ipr{f,b_k}_{L_2} \ell(b_k)
\end{equation} 
for all $f\in F_0$ 
and $\ell\in\Lambda_0$. 
\end{lemma}

(The modifications in case of $d_k(F,L_2)=0$ for $k\ge k_0$ 
are straightforward.) 

\bigskip
\goodbreak

Before we prove this lemma, let us see how it implies Theorem~\ref{thm:general}.

\begin{proof}[Proof of Theorem~\ref{thm:general}]
We want to apply Theorem~\ref{thm:general-H} 
to the Hilbert space~$H$ 
from Lemma~\ref{lemma:emb}.  
In order to extend $\Lambda$ to linear and continuous functionals on $H$, 
let us 
fix a countable dense subset $F_0\subset F$, 
and corresponding $\Lambda_0$ as in Lemma~\ref{lemma:emb}, and 
(formally) define, for each $\ell\in \Lambda_0$, the functionals 
$\overline{\ell}(f)\colon H\to \R$ by 
$\overline{\ell}(f)
=\sum_{k\in\N} \ipr{f, b_k}_{L_2} \ell(b_k)$. 

In order to show the boundedness/continuity of $\overline{\ell}\colon H\to\R$, 
note that 
\[
|\overline{\ell}(f)|
\;\le\; \sum_{k\in\N} |\ipr{f, b_k}_{L_2} \ell(b_k)|
\;\le\; \|f\|_H \,\sqrt{\sum_{k\in \N}\sigma_k^2\, \ell(b_k)^2}
\]
with $\sigma_k^2:=k^{-1/2} d_{\lfloor k/8\rfloor}(F,L_2)$.
Using  
\[
\int_{\Lambda}\sum_{k=1}^n\sigma_k^2\ell(b_k)^2\dint\nu(\ell)
\,\le\, \sum_{k\in \N}\sigma_k^2 \,<\, \infty,
\]
for $n\in\N$, 
we obtain with the monotone convergence theorem that there is $\Lambda'\subset \Lambda$ with $\nu(\Lambda\setminus \Lambda')=0$ such that $\sum_{k\in \N}\sigma_k^2\ell(b_k)^2<\infty$ for all $\ell\in \Lambda'$. 

Thus, for every $\ell\in \Lambda_1:=\Lambda_0\cap\Lambda'$, the functional $\overline{\ell}\colon H\to \R$ is linear and continuous on $H$, and we set $\overline{\Lambda}:=\{\overline{\ell}\colon \ell\in\Lambda_1\}\subset H'$.

We can now apply Theorem~\ref{thm:general-H} to $H$ 
and $\overline{\Lambda}$, where 
we equip $\overline{\Lambda}$ in the natural way with the same measure as $\Lambda$.  
Equation~\eqref{eq:info-ass} then follows from \hyperlink{assumA}{Assumption~(A.3)}.
We obtain, with probability $1-N^{-c}$, that 
 \[\begin{split} 
\Big\|f - A_N'(f)\Big\|_{L_2}
\,&\le\, \sqrt{\frac{1}{n} \sum_{k\ge n} \frac{d_k(F,L_2)}{\sqrt{k}}} 
\cdot 
\|f-P_{32n}f\|_H \\
 \end{split}
\] 
for all $f\in H$, 
if $N\gtrsim_c n\log(n+1)$ 
and $A_N'\colon H\to V_{32n}$ is the algorithm from Theorem~\ref{thm:general-H} 
with iid functionals $\overline{\ell}_i\in\overline{\Lambda}$ according to 
$\rho_{32n}\dint\nu$. 

Now, for the algorithm $A_N$ as described in Theorem~\ref{thm:general}, 
note that $\ell_i\in\Lambda_1$ almost surely and, 
by Lemma~\ref{lemma:emb}, $\overline{\ell}=\ell$ on $F_0$ for all $\ell\in\Lambda_1$. 
This implies that ``$A_N(f)=A_N'(f)$ for all $f\in F_0$'' with probability one. 
We  obtain that 
 \[\begin{split}
\sup_{f\in F_0}\Big\|f - A_N(f)\Big\|_{L_2}
\;&\le\; \frac{4}{\sqrt{n}} \sum_{k\ge n} \frac{d_k(F,L_2)}{\sqrt{k}}
 \end{split}
\]
with probability $1-N^{-c}$ if $N\gtrsim_c n\log(n+1)$, 
where we also used~\eqref{eq:F-H-bound}.
Since $F_0$ is dense in $F$ and
both ${\rm id} \colon F \to L_2$ and $A_n \colon F \to L_2$
are continuous $\nu$-a.e.,
the same is true for $F$. 

\end{proof}

\bigskip
\goodbreak

\begin{proof}[Proof of Lemma~\ref{lemma:emb}]
From~\cite[Lemma~3]{KU21a}, 
using that $d_n(F,L_2)=a_n(F,L_2)$ for all $F\subset L_2$, 
we find 
an ordered orthonormal system $(b_k)_{k\in \N}$ in $L_2$
such that 
\begin{equation*}
 \sup_{f\in F} \Vert f - P_n f \Vert_{L_2}  
\,\le\, 2\, d_{\lfloor n/4\rfloor}(F,L_2)
\end{equation*}
for all $n\in\N$. 
In particular, every $f\in F$ expands into $f=\sum_{k\in\N} \ipr{f, b_k}_{L_2} b_k$ in $L_2$. 

Let $H$ be the Hilbert space in the statement of Lemma~\ref{lemma:emb}, i.e., we complete $\textrm{span}\, \{b_k\}_{k\in \N}$ with respect to the norm
\begin{equation*}
\| f\|_H^2:=\sum_{k\in \N}\frac{|\ipr{f,b_k}_{L_2}|^2}{\sigma_k^2}\qquad \text{with }\sigma_k^2:=\frac{d_{\lfloor k/8\rfloor}(F,L_2)}{\sqrt{k}}, \quad k\in \N.    
\end{equation*}
Then, writing 
$d_k:=d_k(F,L_2)$ and 
$\alpha_n 
\,:=\, \sup_{f \in F} \Vert f - P_{n} f  \Vert_2
$, 
we have 
$\alpha_k\le 2 d_{\lfloor k/4\rfloor}$ and 
obtain for $n\in8\N$ and $f\in F$ that
\[
\begin{split}
	\Vert f - P_{n} f \Vert^2_H 
	\,&=\, \sum_{k > n} \frac{\sqrt{k}\, \abs{\ipr{f,b_k}_{L_2}}^2 }{d_{\lfloor k/8\rfloor}}
	\,\le\, \sum_{ \ell  =0}^\infty \frac{\sqrt{2^{\ell+1}n}}{d_{\lfloor 2^{\ell-2}n \rfloor}} 
	\sum_{k= 2^{\ell}n +1}^{2^{\ell+1}n}  \abs{\ipr{f, b_k}_{L_2}}^2\\
	&\le\, \sum_{ \ell  =0}^\infty \sqrt{2^{\ell+1}n} \cdot \frac{\alpha_{2^{\ell}n}^2}{d_{\lfloor 2^{\ell-2}n \rfloor}}
	\le\, 2 \sum_{ \ell  =0}^\infty \sqrt{2^{\ell+1}n} \cdot d_{\lfloor 2^{\ell-2}n \rfloor} \\
	\,&\le\, 2 \sum_{ \ell  =0}^\infty \sqrt{2^{\ell+1}n} \cdot \frac{1}{2^{\ell-3}n} \sum_{k=\lfloor 2^{\ell-3}n\rfloor+1}^{2^{\ell-2} n} d_k\\
	&\le\, 8\sqrt2 \sum_{ \ell  =0}^\infty \sum_{k=\lfloor 2^{\ell-3}n\rfloor+1}^{2^{\ell-2} n} \frac{ d_k}{\sqrt{k}}
	\,\le\, 12 \sum_{k> \lfloor n/8\rfloor} \frac{ d_k}{\sqrt{k}}.
	\end{split}
\]
In particular, $\lim_{n\to\infty}P_nf= f\in H$.

For~\eqref{eq:ell-cont}, note that, by \hyperlink{assumA}{Assumption~(A.3)}, 
the mapping $I_\Lambda\colon L_2(D,\mu)\to L_2(\Lambda,\nu)$, 
$I_\Lambda f(\ell):=\ell(f)$ is an isometry, 
and hence linear and injective on $L_2(D,\mu)$. 
Thus, 
$\{I_\Lambda b_k\}_{k\in\N}$ is an orthonormal system 
in $L_2(\Lambda,\nu)$, and for every 
$f\in L_2(D,\mu)$
we have
$f^\Lambda=\sum_{k\in\N} \ipr{f, b_k}_{L_2} b_k^\Lambda$ in $L_2(\Lambda,\nu)$, 
where we write $f^\Lambda:=I_\Lambda f$.
Every $f\in F$ satisfies 
\[
 \sum_{k\ge 1} k\, |\ipr{f, b_k}_{L_2}|^2 
 \,=\, \sum_{n\ge 0} \sum_{k>n} |\ipr{f, b_k}_{L_2}|^2
 \,=\, \sum_{n\ge 0} \Vert f - P_n f \Vert_{L_2}^2 
\,<\, \infty,
\]
and hence, by the Rademacher-Menchov theorem, see e.g.~\cite{Me07,S41}, 
we obtain 
\begin{equation}\label{eq:pointwise}
\ell(f)=f^\Lambda(\ell)
=\sum_{k\in\N} \ipr{f, b_k}_{L_2} b_k^\Lambda(\ell)
=\sum_{k\in\N} \ipr{f, b_k}_{L_2} \ell(b_k)     
\end{equation}
for $\nu$-almost all $\ell\in\Lambda$. 
Since $F_0$ is countable, the almost everywhere convergence holds 
simultaneously for all $f\in F_0$, 
i.e., there is $\Lambda_0\subset\Lambda$ with $\nu(\Lambda\setminus \Lambda_0)=0$ and  
\eqref{eq:pointwise} for all $f\in F_0$ and $\ell\in\Lambda_0$. 
In particular, $\ell(b_k)$ is well-defined for $\ell\in\Lambda_0$ and $k\in\N$.

\end{proof}

\bigskip
\goodbreak

\subsection{Approximation in general norms}
\label{subsec:general}

We will now present the techniques from~\cite{KPUU23} to transfer the 
$L_2$-error bounds to more general seminorms. 
However, the results come with additional restrictions on the used approximation 
spaces.

\medskip

\hypertarget{assumB}{}
\noindent\textbf{Assumption B.} \ 
We consider the following setting. 
\begin{itemize}
 \item[(B.1)] 
    $G$ is a seminormed space which contains $F$, and $G\cap L_2$ is complete w.r.t.~the natural seminorm $\|\cdot\|_*:=\|\cdot\|_G+\|\cdot\|_{L_2}$. 
   If two functions from $G$ are equal $\mu$-almost everywhere, then their seminorm in $G$ is the same.
   \item[(B.2)] $(V_n)_{n=1}^\infty$ is a sequence of subspaces of $G \cap L_2$, respectively of dimension $n$.
\end{itemize}

\medskip

Assumption \hyperlink{assumB}{(B.1)} is satisfied, e.g., if 
$F$ is a compact subset of $C(D)$, 
where $D$ is a compact domain, $\mu$ is a finite measure on $D$,
and $G=L_q(\mu)$ for $1\le q\le\infty$. Note that 
for probability measures, and $q\le 2$, the above presented upper bounds for $L_2$-approximation directly transfer to $L_q$-approximation via $\|\cdot\|_{L_q}\le \|\cdot\|_{L_2}$, and in also to the integration problem. 
See~\cite{KPUU23} for the discussion of more general norms, in particular $G=C(D)$. 

We introduce the quantities 
\begin{equation}\label{eq:B}
B_n \;:=\;B(V_n,G) 
\;:=\; \sup\limits_{f \in V_n, \, f\neq 0} \frac{\|f\|_G}{\|f\|_2} 
\end{equation}
and
\begin{equation*}
\alpha_n \,:=\, \alpha(V_n,F) \,:=\, \sup_{f \in F} \Vert f - P_{V_n} f  \Vert_2
\;=\; \sup_{f \in F}\, \inf_{g\in V_n}\Vert f - g  \Vert_2. 
\end{equation*}

Note that $B_n$ corresponds to the (inverse) of the Christoffel function if $G=L_{\infty}$. Further, for optimal subspaces the quantity $B(V_n,G)$ is equal to 
\[
\inf_{\dim(V_n)=n}\sup\limits_{f \in V_n, \, f\neq 0} \frac{\|f\|_G}{\|f\|_2}
\,=\, b_{n-1}({\rm id}\colon G\to L_2)^{-1},
\]
where $b_n$ denotes the \emph{Bernstein $n$-width}, see \cite[Ch. II]{Pinkus85}.

\medskip
\goodbreak 

The following lemma is easy to prove. We refer to \cite[Lemma~10]{KPUU23}.

\begin{lemma}\label{lemma:lift}
 Let \hyperlink{assum}{Assumption~B} hold. 
 For any $n\in\N$, any mapping $A\colon F \to V_n$ and all $f\in F$,  we have
 \begin{equation*}
  \Vert f - Af\Vert_{G}\,\le\, 
	2 \sum_{k > \lfloor n/4 \rfloor 
    } \frac{ \alpha_k \, B_{4k}}{k} 
  \,+\, B_n \cdot \Vert f-Af\Vert_2. 
 \end{equation*}
 \end{lemma}
 
\medskip

This bound can now be used in combination with the algorithm and sampling strategy 
from above. We only state a special case, 
and refer to \cite{KPUU23} for the general result, the proof and a discussion 
of some cases where this leads to sharp results 
(up to the logarithmic oversampling).

\begin{thm}\label{thm:general-G}
Let Assumptions \hyperlink{assumA}{A} and \hyperlink{assumB}{B} hold and let 
$\alpha,\beta,\gamma$ and $\delta$ be real parameters with 
$\alpha > \max\{\beta, 1/2\}$.
If
\[
B(V_n,G) \;\lesssim\; n^\beta (\log n)^{\delta}
\quad \text{ and } \quad 
\alpha(V_n,F) \,\lesssim\, n^{-\alpha} (\log n)^\gamma
\]
then
\[
e_N^{iid}(F,L_2,\rho_{n}\cdot\dint\nu) 
\;\lesssim\; n^{-\alpha+\beta} (\log n)^{\gamma+\delta}
\]
with probability $1-N^{-\alpha}$ if $N\gtrsim n \log n$, 
where $\rho_n$ is as in Theorem~\ref{thm:general}.
\end{thm}

\medskip

Note that this includes the (sharp) corollary of Theorem~\ref{thm:general} 
for $G=L_2$: In this case $\beta=\delta=0$, 
and with the choice of $(V_n)$ such that 
$\alpha(V_n,F)\le 2 d_{\lfloor n/4\rfloor}(F,L_2)$, we obtain that 
$d_n(F,L_2)\asymp n^{-\alpha} (\log n)^\gamma$ with $\alpha>1/2$ implies 
\[
e_{n\log n}^{iid}(F,G,\rho_{n}\cdot\dint\nu) 
\;\lesssim\; d_n(F,L_2)
\;=\; a_n(F,L_2)
\]
with probability $1-n^{-\alpha}$ whenever 
\hyperlink{assumA}{Assumption~A} holds.

\medskip

\subsection{Sharpness for $L_2$-approximation}
\label{sec:sharpness}

Let us note that Theorem~\ref{thm:general} may be applied to the setting in Section~\ref{sec:hilbert} and in particular Theorem~\ref{thm:general-H}, where $d_k(B_H,L_2)=\sigma_{k+1}$. A qualitative comparison gives bounds on $e_{N}^{iid}(B_H,L_2,\Lambda)$ whp, where $N\gtrsim n\log n$, which are of order
\begin{equation*}
\frac{1}{\sqrt{n}} \sum_{k> n} \frac{\sigma_k}{\sqrt{k}}\quad\text{vs.}\quad
\frac{1}{\sqrt{n}} \sqrt{\sum_{k> n} \sigma_k^2},
\end{equation*}
where apart from the common prefactor the first expression is the Lorentz $\ell_{2,1}$-norm and the second the $\ell_2$-norm of the tail. In general, we have $\ell_{2,1}\subset \ell_{2,2}=\ell_{2}$, see e.g. \cite[Prop. 2.1.10]{Pie87}. That is, if the second sum converges, then so does the first. Let us give some more details on the convergence.

In Section~\ref{sec:improve}, we shall be interested in sequences 
$(\sigma_n)$ with 
$\sigma_n\asymp n^{-\alpha}(\log n)^{-\beta}$ for $\alpha,\beta>0$, which is the case for $n$-widths of Sobolev embeddings, see e.g. \cite[Ch. VII]{Pinkus85} or \cite[Ch 5]{Te18}. In this case, 
\[
\frac{1}{\sqrt{n}} \sum_{k> n} \frac{\sigma_k}{\sqrt{k}}
\;\asymp_{\alpha,\beta}\;
\begin{cases}
\sigma_n &\text{if }\alpha>1/2,\\
\sigma_n\log n&\text{if }\alpha=1/2\text{ and }\beta>1,\\
\end{cases}
\]
and
\[
\frac{1}{\sqrt{n}} \sqrt{\sum_{k> n} \sigma_k^2}
\;\asymp_{\alpha,\beta}\;
\begin{cases}
\sigma_n &\text{if }\alpha>1/2,\\
\sigma_n\sqrt{\log n}&\text{if }\alpha=1/2\text{ and }\beta>1/2,\\
\end{cases}
\]
and in the remaining cases the sum does not converge. So, the second bound is slightly smaller for poly-log decay. 

We obtain that $n\log n$ iid measurements are asymptotically optimal whp, i.e., $e_{n\log n}^{iid}(B_H,L_2,\Lambda)\asymp \sigma_n$, 
provided that the singular numbers decay at a rate $n^{-\alpha}$ with $\alpha>1/2$. 
At $\alpha=1/2$ we lose logarithmic factors.  But also these bounds are sometimes sharp up to the oversampling, see~\cite{DKU23,KV23}.
In certain cases, the oversampling can be removed, see Section~\ref{sec:improve}. 
In others, it is necessary, 
as we shall see in the following.

For this, assume that $H$ is as in Theorem~\ref{thm:general-H}
and 
$\Lambda
:=\{\ipr{\cdot,h}\colon h\in \B\}$, 
where $\B$ is some orthonormal basis of $L_2$.
Note that \hyperlink{assumA}{Assumption~(A.3)} is fulfilled  
if $\nu$ is the counting measure on $\Lambda$. 
Theorem~\ref{thm:general-H} now implies that 
$n \log n$ random coefficients w.r.t.~$\B$ are as powerful as 
the optimal information for Hilbert spaces. 

Now, if $\B=\{b_1,b_2,\dots\}$ is the optimal basis as in \eqref{eq:H-norm}, then the density with respect to the counting measure $\nu$ on $\{\ipr{\cdot,b_k}\colon k\in \N\}$ is given by
\[
\rho_n(\ipr{\cdot,b_\ell})
\,=\, \frac{1}{2}
\begin{cases}
	\frac{1}{n} & \text{if } \ell\le n,\\
	\frac{\sigma_{\ell}^2}{\sum_{k>n}\sigma_k^2} & \text{if } \ell> n.
\end{cases}
\]
Further, by virtue of the coupon collectors theorem, 
the asymptotics $n \log n$ cannot be improved. More precisely, one needs on average $n\log n$ samples drawn according to $\rho_n$ (or any other distribution) to guarantee that the first $n$ Fourier coefficients are evaluated. For completeness, we give a formal statement below. A similar effect occurs if we approximate Sobolev functions using function samples, see Section~\ref{sec:sobolev}. This shows that in general Theorem~\ref{thm:general-H} and thus Theorem~\ref{thm:general} cannot be improved. 

\begin{prop}\label{pro:fourier-lower}
	Assume that $H$ is as in Theorem~\ref{thm:general-H} and $\Lambda :=\{\ipr{\cdot,h}\colon h\in \B\}$, where $\B$ is some orthonormal basis of $L_2$. Let $\nu$ be any probability measure on $\Lambda$. If $N/n\log n\to 0$, then for any $\alpha>0$ we have 
\[
e_N^{\iid}(B_H,L_2,\nu)\ge \sigma_{\lfloor \alpha n\rfloor}
\]
whp for all large enough $n$.
\end{prop}

If the sequence $(\sigma_n)$ decays fast enough, i.e., if $\lim_{K\to \infty}\sup_{n\in\IN}\frac{\sigma_{Kn}}{\sigma_n}=0$ such as for polynomial decay, then under the conditions in the proposition we deduce that for any $C>0$ 
\[
e_N^{\iid}(B_H,L_2,\nu)\ge C\sigma_n
\]
whp for all $n$ large enough. Thus, for sampling Fourier coefficients, $N\ll n\log n$ iid functionals are whp worse by an arbitrary factor than $n$ optimal functionals.

\medskip

This shows that for some classes of information that clearly contain optimal information, we need a logarithmic oversampling. 
This might not come as a surprise. However, we will see in Section~\ref{sec:improve} that this is not true in general, depending on the classes of information and inputs. 
This leads to the following question.

\begin{OP}
What are conditions on $\Lambda$ and $H$ such that the conclusion of Proposition~\ref{pro:fourier-lower} holds?
\end{OP}

\medskip

This should be compared to Theorems~\ref{thm:sob-main} and~\ref{thm:gauss}, and Open Problem~\ref{op:3} below.

\medskip

\begin{proof}[Proof of Proposition~\ref{pro:fourier-lower}]
We identify $\Lambda$ with $\IN$ and thus $N$ iid functionals $\ell_1,\dots,\ell_N$ in $\Lambda$ sampled according to $\nu$ correspond to $N$ numbers sampled randomly from $\IN$.

Since $N/n\log n\to 0$, for any $n$ large enough we have $N\le \frac12 n\log n$. By the coupon collector's problem, the probability that we miss one of the numbers in $\{1,\dots,n\}$, say $i$, is at least $1-e^{-\sqrt{n}}$. For this, combine the classical limit law in \cite{ER61} with the fact that equidistribution on $\{1,\dots,n\}$ stochastically needs the least amount of coupons, see \cite[p.52]{BH97}. In this case, the $i$-th Fourier coefficient is not measured and any algorithm $A_N$ lacking this information has error at least $\sigma_i\ge \sigma_n$ (since $\pm\sigma_i b_i\in B_H$, it has to return zero for these functions). 

Thus, $e_N^{\iid}(B_H,L_2,\nu)\ge \sigma_n$ in this case. By replacing $n$ with $\lfloor \alpha n\rfloor$ for $\alpha>0$ we get that with probability at least $1-e^{-\sqrt{\alpha n}}$ for all large enough $n$ we have 
\[
e_N^{\iid}(B_H,L_2,\nu)\ge \sigma_{\lfloor \alpha n\rfloor}.
\]
\end{proof}

\bigskip
\goodbreak
\newpage

\section{Applications and improvements}
\label{sec:improve}

The results above show that $n \log n$ pieces of iid information are 
often as valuable as optimal information, and we have even seen 
that this cannot be improved in general. 
However, there are cases where the logarithmic oversampling factor can be removed and iid information is asymptotically optimal. 
We report on two instances. Namely, 
$L_q$-approximation in (isotropic) Sobolev spaces $W^s_p$ with~$p>q$ if iid uniform samples are used, and 
$L_2$-approximation if the iid information 
is \emph{Gaussian}, in both a linear and a nonlinear setting.
We still do not completely understand 
what makes these settings special in this respect,  
and add some open problems related to them. 

We start with a more detailed explanation of 
the general case of 
standard information, i.e., function evaluations. 
In particular, we discuss a natural limitation 
of this class of information.

\subsection{Random function evaluations}
\label{sec:std}

Approximation of (regular or smooth) functions
using function evaluations 
was the main motivation and application of the results introduced in Section~\ref{sec:general}. 
Clearly, for $\Lambda=\Lambda^\std$, 
\hyperlink{assumA}{Assumption~(A.3)} 
is obvious by identifying $\nu$ with $\mu$ (from $L_2(\mu)$) 
by $\nu(\{\delta_x\colon x\in M\})=\mu(M)$ for $\mu$-measurable $M\subset D$. For applying the results we would like point evaluation to be continuous on $H$.

Especially Hilbert spaces of this type are of interest and intensively studied: \\
A Hilbert space $H$ of functions on a set $D$ 
is called a \emph{reproducing kernel Hilbert space (RKHS) on $D$} if 
point evaluation $\delta_x\colon H\to\IR$
is a continuous functional for all $x\in D$,
i.e., $H$ is a RKHS if and only if $\Lambda^{\std}\subset H'$.
The crucial property of RKHS is 
the existence of a (reproducing) kernel $K\colon D\times D\to\R$ such that 
$f(x)=\ipr{f,K(x,\cdot)}_H$ for all $f\in H$ and $x\in D$. If $H\hookrightarrow L_2(\mu)$ is compact, the kernel characterizes $H$ in the sense of \eqref{eq:H-norm} being equivalent to 
\begin{equation} \label{eq:kernel}
K(x,y) \,=\, \sum_{k=1}^{\infty} \sigma_k^2\, b_k(x)\, b_k(y), 
\qquad x,y\in D, 
\end{equation}
where, again, 
$(\sigma_k)\in \ell_2$ is a non-increasing sequence and $\{b_k\}_{k=1}^\infty$
is an orthonormal system in~$L_2$, see e.g. \cite[Thm. 3.1]{SS12}. We refer to \cite{Aro50} for more on RKHS.

Let us note that the results from the last section do not require \emph{every} 
$\delta_x$ to be in $H'$; almost all would suffice. 
However, this is not a major restriction, as we can always pass to the 
(full-measure) subset $D_0\subset D$, where function evaluation is continuous. This does not change the $L_2$-error, or the random sampling. 

Again, we obtain from Theorem~\ref{thm:general-H}, see~\cite{KU21}, that there is some sampling density~$\rho_n$, 
such that $n\log n$ random sampling points are enough for a near-optimal bound. 
In particular, the algorithm 
\[
A_{N}(f) \,=\, 
\underset{g\in V_n}{\rm argmin}\, \sum_{i=1}^N 
\frac{\vert g(x_i) - f(x_i) \vert^2}{\rho_n(x_i)}
\]
for suitable $V_n$ and $\rho_n$ from~\eqref{eq:density} (with $\ell=\delta_x$) satisfies 
\[
e(A_N,B_H,L_2) 
\;\le\; 
\sqrt{\frac1n\sum_{k> n} 
c_k(B_H,L_2)^2},
\]
with probability $1-N^{-c}$, 
whenever $N\gtrsim_c n\log n$, 
and $x_1,\dots,x_N$ are iid w.r.t.~$\rho_n \dint\mu$.

It has already been observed in \cite{HNV08} 
that the square-summability of $(c_n)$ is a necessary condition for a general comparison with $(g_n)$. 
In fact, it is shown in \cite{KV23} that there exists some $c>0$ 
such that 
\[
g_{\lfloor cn\rfloor}(B_H,L_2) \;\ge\; 
\sqrt{\frac1n\sum_{k> n} 
c_k(B_H,L_2)^2}
\]
for some (simple, univariate) Hilbert space $H$ and all $n$, see~\eqref{eq:gn}.
That is, the upper bound is optimal, up to the logarithmic oversampling factor, and 
even for optimal function evaluations the square-summability of $(c_n)$ is necessary. 
In fact, for many important examples, as the Sobolev spaces discussed below, this summability corresponds to the embedding into $C(D)$. It is therefore necessary to work with function evaluations and only a weak restriction. 
We will see in Section~\ref{sec:gauss-lin} that the same restriction appears for Gaussian information.

So, in the case of standard information in a RKHS, 
iid information is optimal 
up to the logarithmic oversampling factor. 
Similar results exist for more general classes, see~\cite{KU21a,DKU23}. 

\begin{rem}[$g_n\neq c_n$?]
It is easy to find examples where 
standard information is as powerful as arbitrary linear information. 
For this consider the (pathological) example 
of $D=\N$ and a RKHS $H\subset \ell_2$ on $\N$ as in \eqref{eq:H-norm} with norm $\|f\|_H=\sum_{k=1}^{\infty}f_k^2\sigma_k^{-2}$, for $f=(f_1,f_2,\dots)$, i.e., the ONB $\{b_k\}_{k=1}^{\infty}$ is given by the canonical basis of $\ell_2$. Clearly, function evaluation is the same as computing coefficients w.r.t.~the (optimal) basis $(b_k)$, and so 
$g_n(B_H,L_2)=c_n(B_H,L_2)$ 
in this case. 
\end{rem}

\bigskip
\goodbreak

However, when we turn to approximation in other norms, 
then it seems that, so far, no general ``for all $H$'' comparison has been observed, 
and additional conditions appear; 
often involving the quantity $B(V_n,G)$ from~\eqref{eq:B}. 
Let us only discuss the case of \emph{uniform approximation} $G=L_\infty$, and refer to~\cite{KPUU23} for generalizations. 
In this case, we have 
\[
B(V_n) \,:=\, B(V_n,L_{\infty})
\,=\, \sup\limits_{f \in V_n, \, f\neq 0} \frac{\|f\|_\infty}{\|f\|_2}
\,=\, \Big\|\sum_{k=1}^{n}|b_k|^2\Big\|_{\infty}^{1/2},
\]
where $\{b_1,\dots,b_n\}$ is an arbitrary $L_2$-orthonormal basis of $V_n$.

As Theorem~\ref{thm:general-G} shows, 
bounds on $B(V_n)$, together with good $L_2$-approximation properties of $V_n$, 
leads to a corresponding upper bound on the error of $A_N$ from above. 
However, for finite $\mu$, we have
\[
B(V_n,L_{\infty})
\,\ge\, \Big(\frac{1}{\mu(D)}\sum_{k=1}^{n}\|b_k\|_2^2\Big)^{1/2}
=\sqrt{\frac{n}{\mu(D)}},
\]
which shows that we lose at least a factor of $\sqrt{n}$ compared to the $L_2$-error.

For RKHS on finite measure spaces with $B(V_n)\asymp\sqrt{n}$ for the optimal subspaces $V_n$ we have the following result which is  essentially Theorem~6 in~\cite{KPUU23} and 
was obtained independently in~\cite{GW24}. 

\begin{thm}\label{thm:uniform-H}
There are
absolute constants 
$b,c\in\N$ 
such that the following holds.
Let $\mu$ be 
a finite measure
on a set $D$ 
and
$H\hookrightarrow L_2$ be a 
reproducing kernel Hilbert space with kernel
as in~\eqref{eq:kernel}
with $\sigma_{2n}\gtrsim \sigma_{n}$ and 
\[
\sup_n 
\Big\Vert \frac{1}{n}\sum_{k=1}^{n} \left| b_k \right|^2\Big\Vert_{\infty} \,<\, \infty.
\]
Then, $H \hookrightarrow  L_\infty$ and 
the unweighted least squares method 
\begin{equation} \label{eq:alg-u}
A^u_{N}(f) \,:=\, 
\underset{g\in V_n}{\rm argmin}\, \sum_{i=1}^N 
\vert g(x_i) - f(x_i) \vert^2
\end{equation}
with $V_n={\rm span}\{b_1,\dots,b_n\}$ and $x_1,\dots,x_N\overset{\rm iid}{\sim}\mu$ satisfies 
\[
 e(A_N^u,B_H,L_\infty) 
 \;\lesssim\; c_n(B_H,L_\infty)
 \;=\; a_n(B_H,L_\infty)
\]
with probability $1-N^{-c}$, 
whenever $N\gtrsim_c n\log n$.
\end{thm}

Note that no decay condition on $(c_n)$ and no sampling density depending on a basis is needed - the unweighted least squares algorithm is \emph{universal}.  
We will see in 
the following section 
that for $L_q$-approximation with $q<2$ the logarithmic oversampling 
can be removed in the case of Sobolev spaces.

\medskip

The assumptions of Theorem~\ref{thm:uniform-H} hold for example 
if the basis $\{b_k\}$ is bounded which is the case for the trigonometric system or the Chebychev system or (Haar) wavelets, 
if $\mu$ is their corresponding orthogonality measure. This includes many interesting RKHSs such as certain Sobolev spaces.  See~\cite{KPUU23} 
for several examples. 

\medskip

Still, it is not clear how far this result can be extended.

\begin{OP}
Find necessary conditions on $H$ such that 
the conclusion of Theorem~\ref{thm:uniform-H} holds. 
Moreover, find a variant for more general $F\hookrightarrow L_\infty$. 
\end{OP}

\medskip

Note that Theorem~\ref{thm:uniform-H} is only implicitly contained in~\cite{GW24,KPUU23} as both papers work directly with the optimal \emph{subsampled} algorithm from~\cite{DKU23}. (We discuss this shortly in Section~\ref{subsec:sub}.) However, since the proof is based on a variant of Lemma~\ref{lemma:lift}, see Section~3.2 of~\cite{KPUU23}, 
it is apparent that one may also work with 
the algorithm $A_N$ from Theorem~\ref{thm:general-H}.

\medskip 

Another ingredient in the above theorem is the next result that allows for removing the 
weights from algorithm and sampling. 
We state it for future reference. 
 
\begin{prop}\label{prop:christoffel}
Let $H$, $\mu$ and $\Lambda=\Lambda^{\std}$ be as in Theorem~\ref{thm:uniform-H}. Then the conclusion of Theorem~\ref{thm:general-H} and consequently Theorems~\ref{thm:general} and~\ref{thm:general-G} continue to hold for the sampling density $\varrho_n\equiv \mu(D)^{-1}$, $n\in\IN$.
\end{prop}

Note that constant weights can be replaced by 1 in~\eqref{eq:alg} and thus the algorithm is an unweighted least squares method as in \eqref{eq:alg-u}.

\bigskip

\begin{proof}[Proof of Proposition~\ref{prop:christoffel}]
We consider Theorem~\ref{thm:general-H} for $H$, $\mu$ and $\Lambda^{\std}$ as in the statement of the proposition. Then $\Lambda^{\std}\subset H'$ and \eqref{eq:info-ass} hold. The sampling density enters the proof of Theorem~\ref{thm:general-H} in the estimate $\|y_i\|_2^2\le 2n$. 
 Using the density $\rho\equiv\frac{1}{\mu(D)}$, and therefore $w_i=\mu(D)$, instead, 
 we see that $\|y_i\|_2^2\le 2C \mu(D)\, n$ a.s.~is implied by $|\rho_n(x)|\le C$ for $\mu$-almost all $x\in D$ and all $n$. 
 We can therefore apply Lemma~\ref{lemma:concentration} with the corresponding $R$. \\

	Regarding the first summand in \eqref{eq:density}, we have for $\mu$-almost all $x\in D$ that
\[
\frac{1}{n}\sum_{k=1}^{n}|b_k(x)|^2
\le \Big\|\frac{1}{n}\sum_{k=1}^{n}|b_k|^2\Big\|_{L_{\infty}(\mu)}
=\frac{1}{n}B(V_n,L_{\infty}(\mu))^2\lesssim 1.
\]
To investigate the second summand in \eqref{eq:density}, let $n\in\IN$ and pick $\ell\in\IN$ such that $2^{\ell}\le n\le 2^{\ell+1}$. Then, for 
$\mu$-almost all $x\in D$,  
\begin{align*}
\sum_{k\ge n}\sigma_k^{2}|b_k(x)|^2
\le \sum_{k\ge 2^{\ell}}\sigma_k^{2}|b_k(x)|^2
&\le \sum_{i=\ell}^{\infty}\sigma_{2^i}^2\sum_{k=2^i}^{2^{i+1}-1}|b_k(x)|^2\\
&\lesssim \sum_{i=\ell}^{\infty}\sigma_{2^i}^2 2^{i+1}
\lesssim\sum_{i=\ell}^{\infty}\sum_{k=2^{i-1}}^{2^{i}-1}\sigma_{k}^2 
\lesssim \sum_{k\ge n/4}\sigma_k^2.
\end{align*}
It remains to use $\sum_{k\ge n/4}\sigma_k^2\lesssim \sum_{k\ge n}\sigma_k^2$ which follows from assuming $\sigma_{2n}\gtrsim \sigma_n$. \\
\end{proof}

\bigskip

\subsection{Sharp results for Sobolev spaces}\label{sec:sobolev}

In this section we take a closer look at $L_q$-approximation in isotropic Sobolev spaces 
for which we have a characterization of the quality of (random) samples due to \cite{KS23, KS22a} which implies asymptotic optimality of $n$ or $n\log n$ iid measurements depending on the parameters involved. There are also generalizations to similarly structured isotropic function spaces such as Hölder, Triebel-Lizorkin or Besov spaces.

On a domain $D\subset\R^d$ (i.e., an open and nonempty set), equipped with the Lebesgue measure, the Sobolev space of smoothness $ s\in\mathbb{N} $ and integrability $ 1\leq p\le \infty $ is given by
\[
W^s_p(D):=\Big\{f\in L_p(D)\colon \|f\|_{W^s_p(D)}:=\Big(\sum_{|\alpha|\leq s} \|D^{\alpha}f\|_{L_p(D)}^p\Big)^{1/p}
 <\infty
\Big\},
\]
where the sum is over all multi-indices $ \alpha\in \mathbb{N}_0^d $ with $ |\alpha|=\alpha_1+\ldots+\alpha_d \le s $ and $D^\alpha f = \frac{\partial^{|\alpha|}}{\partial x_1^{\alpha_1}\cdots\partial x_d^{\alpha_d}} f$ denotes a weak partial derivative of order $|\alpha|$. In the following, we denote by $B_p^s(D)$ the unit ball of $W^s_p(D)$. 

Sobolev functions from $W^s_p(D)$ do have well-defined function values if the embedding $W^s_p(D)\hookrightarrow C_b(D)$ into the bounded continuous functions holds, i.e., if 
\begin{equation} \label{eq:embedding}
	s>d/p \quad \text{if }1<p\le \infty \quad \text{or}\quad s\ge d \quad \text{if } p=1,
\end{equation}
and $D\subset \IR^d$ is a bounded Lipschitz domain, and if $s>d/p$ the embedding is compact,
see, e.g., \cite[Sec.~1.4.5]{Maz85}.  Then the sampling numbers, i.e.,
the minimal 
worst-case errors based on function values, are 
known to satisfy 
\begin{equation} \label{eq:minrad-app}
g_n(B_p^s,L_q)
 \,\asymp\, n^{-s/d+(1/p-1/q)_+},
\end{equation}
where $(x)_+=\max\{0,x\}$.

These asymptotics are classical for special domains like the cube and have been obtained with linear algorithms, see e.g. \cite{NT06} and the references therein. 

Let us apply the general results from above in the special case of $p=2$. Then $W^s_2$ is a Hilbert space and since the embedding $W_2^s\hookrightarrow C_b$ is compact and $D$ is bounded, also $W_2^s\hookrightarrow L_2$ is compact. By the spectral theorem, $W^s_2$ is of the form \eqref{eq:H-norm} with
\[
\sigma_{n+1}=c_n(B_2^s,L_2)\asymp n^{-s/d},
\]
see e.g. \cite[Thm. 26]{NT06}. Since $s>d/2$, Theorem~\ref{thm:general-H} gives that whp $N\asymp n\log n$ iid points sampled according to the density $\rho_{32n}$ are as powerful as $n$ optimal points. In order to apply Proposition~\ref{prop:christoffel} and in particular to conclude the same result with constant sampling density, it is sufficient to have $B(V_n)\lesssim \sqrt{n}$ for (almost) optimal subspaces. 

This is for example the case if the domain is a compact Riemannian manifold $M$ of dimension $d$, where the eigensystem of the Laplace-Beltrami operator provides such subspaces. Combining Proposition~\ref{prop:christoffel} with Corollary~31 in \cite{KPUU23} gives that the unweighted least squares method $A_N^u$ from  \eqref{eq:alg-u} using $N\asymp n\log n$ iid random points sampled according to the normalized uniform measure $\mu_M$ on $M$ achieves
\[
e_{n\log n}^\iid(B_2^s,L_q,\mu_M) \;\lesssim\; g_n(B^s_2,L_q)
\]
whp for all $1\le q\le \infty$. We will see in the following that the logarithmic oversampling is necessary if $q\ge 2$ and can be removed otherwise if we use a particular ``localized'' least squares method.

For simplicity, in the remainder of this section, the domain $D$ will be a bounded convex domain (and in particular Lipschitz) and we refer to \cite{KS22a} for (almost) analogous results on manifolds. We will suppress $D$ in the notation.

Given a point set $\Pn=\{x_1,\dots,x_n\}\subset D$ we identify it with the corresponding evaluations. The following characterization of the $n$-th minimal error of iid information, that is, iid points sampled according to the uniform measure $\mu_D$ on $D$, is taken from \cite[Thm. 2]{KNS22} (see \cite[Cor.~2]{KS23} for the original result) and conjectured already in \cite{HKNPU20}, where the case $d=s=1$ has been obtained.

\begin{thm}\label{thm:sob-ran}
Let $1\le p,q \le \infty$ and $s\in\mathbb{N}$ as in \eqref{eq:embedding}. Then, 
\[
\E\, e_n^{iid}\big(B_p^s,L_q,\mu_D\big) 
  \,\asymp\, \begin{cases}
	 g_{n/\log n}\big(B^s_p,L_q\big) & \text{if } q\ge p,\\
	g_{n}\big(B^s_p,L_q\big) & \text{if } q< p\vphantom{\Big|}.
	\end{cases}
\]
\end{thm}

\medskip

Let us note that this result also holds with high probability. 
The following questions are obvious:

\begin{OP}\label{op:3}
Do the bounds of Theorem~\ref{thm:sob-ran} 
also hold for general classes~$F$? 
In particular, under which conditions on $q$ and $F\subset L_q(\mu)$ do we have 
asymptotic optimality of iid function evaluations, i.e., 
$\E\, e_n^{iid}(F,L_q,\mu) 
 \,\asymp\, g_{n}(F,L_q)$? 
Moreover, is logarithmic oversampling necessary, i.e., do we have 
$\E\, e_n^{iid}(B_H,L_2,\mu) 
 \,\asymp\, g_{n/\log n}(B_H,L_2)$ 
 for any RKHS $H$, 
and $\E\, e_n^{iid}(F,L_\infty,\mu) 
 \,\gtrsim\, g_{n/\log n}(F,L_\infty)$ 
 for more general $F\subset L_\infty$?
\end{OP}

\medskip

\begin{rem}
At this point, it seems worthwhile noting that in \cite{NT06} the authors also 
concluded that, if one restricts to linear methods, linear information can be asymptotically better than standard information if and only if $p<2<q$. 
\end{rem}

The algorithm achieving the upper bound in Theorem~\ref{thm:sob-ran} is linear and is based on the \emph{moving least squares} method applied to cones adapted to local density of the sampled point set. For more details we refer to \cite{KNS22} and \cite[Ch. 4]{Wen04}.

In order to describe the algorithm, we introduce a geometric regularity condition on the domain.  We say that a set $D\subset \mathbb{R}^d$ satisfies an \emph{interior cone condition} with radius $r>0$ and angle $\theta\in (0,\pi/2)$ if, for all $x\in D$, there is a direction $\xi(x)\in \mathbb{S}^{d-1}$ such that the cone
\[
C(x,\xi(x),r,\theta):=\left\{x+\lambda y\colon y\in \mathbb S^{d-1}, \langle y, \xi(x)\rangle \geq \cos\theta,\lambda\in [0,r]\right\}
\]
with apex $x$ is contained in $D$. Convex sets satisfy this condition and also bounded Lipschitz domains, see \cite{KS23} for proofs and references. Additionally, we can and do assume that $\theta\le \pi/5$ and that $\xi$ depends continuously on $x$ for almost all $x\in D$.

In the following we will describe the algorithm for a fixed point set $P=\{x_1,\dots,x_n\}\subset D$. We can later insert any realization of a random point set. We shall assume that $P$ is sufficiently dense in $D$.

Given $f\in C_b(D)$ and $x\in D$ we approximate $f(x)$ by
\[
A_P f (x) \,:=\, 
	\argmin_{v\in V_m}\, \sum_{y\in P\cap K_P(x)} w(x,y)\, |f(y)-v(y)|^2,
\]
where $V_m$ is the space of real polynomials of degree at most $m=\lceil s\rceil$,  $K_P(x):=C(x,\xi(x),r_P(x),\theta)$ and the radius $r_P(x)>0$ 
is minimal
such that there are sufficiently many points in $K_P(x)$ to reconstruct all polynomials in $V_m$. 
Further,
the weight function takes the form $w(x,y)=\Phi(x-y)$ where $\Phi$ is supported in $B_2^n(0,\delta)$ and positive on $B_2^n(0,\delta/2)$, 
where $\delta$ depends on $P\cap K_P(x)$.

Thus for evaluating the approximant $A_Pf$ at $x\in D$ we solve a weighted least squares problem depending on the density of points around $x$; 
hence, the terminology ``moving least squares''.

\begin{OP}
Is there an unweighted least squares algorithm such that the bounds of Theorem~\ref{thm:sob-ran} hold? 
\end{OP}

By Proposition~\ref{prop:christoffel} we know that this is the case for $q\ge 2=p$, at least for manifolds.

Note that Theorem~\ref{thm:sob-ran} holds in fact on all domains satisfying the interior cone condition, see \cite[Thm. 2]{KNS22}. However, for bounded convex domains there is a convenient characterization of the radius of information which explains why random points are sometimes optimal and sometimes not. 

 To this end, introduce the covering radius
$h_{\Pn,D}
:=\sup_{x\in D}\dist(x,\Pn)$ 
which is the supremum of the distance function
\begin{equation*}
\dist(\cdot, \Pn)\colon \mathbb{R}^d\to [0,\infty), 
\qquad \dist(x,\Pn):=\min_{y\in \Pn}\|x-y\|_2
\end{equation*}
to the $n$-point sampling set $\Pn\subset D$. 
Although commonly used, the covering radius is insufficient to characterize the power of information as the following result taken from \cite[Thm.~0.1]{KS23} shows.

\bigskip
\goodbreak

\begin{prop}\label{thm:sob-main}
	Let $1\le p,q \le \infty$ and $s\in\mathbb{N}$ as in \eqref{eq:embedding}. For any point set $\Pn\subset D$, we have 
	\begin{align*}
	&r\big(\Pn,B^s_p, L_q\big)\, \asymp \,
	\begin{cases}
\big\|\dist(\cdot, \Pn)\big\|_{L_{\infty}(D)}^{s-d(1/p-1/q)} & \text{if } q\ge p,\vphantom{\bigg|}\\
\big\|\dist(\cdot, \Pn)\big\|_{L_{\gamma}(D)}^s & \text{if } q<p,
\end{cases}
	\end{align*} 
	where $\gamma=s(1/q-1/p)^{-1}$ and the implicit constants are independent of $\Pn$.
\end{prop}

Thus, the quality of a point set is asymptotically determined by the radius of the largest hole amidst the points if $q\ge p$ and by an average of the distance to the point set if $q<p$. Partial results have been obtained in \cite{HKNPU20,NWW04,NT06,P98,Su79}. 

Theorem~\ref{thm:sob-main} is a tool to analyze the asymptotic optimality of arbitrary (sequences of) point sets and, in particular, random or typical ones. To compare, the optimal behaviour of the $L_{\gamma}$-norm of the distance function is
\begin{equation*} 
\inf_{\#\Pn \le n} \|\dist(\cdot, \Pn)\|_{L_{\gamma}(D)} \,\asymp\, n^{-1/d} 
\qquad \text{for every} \quad
0<\gamma\leq \infty.
\end{equation*}
By Theorem~\ref{thm:sob-main}, point sets attaining this rate yield the upper bound in \eqref{eq:minrad-app}. 

For uniform random points on a bounded convex domain, that is, iid points distributed according to $\mu_D$, it is known that the average hole size is on average of optimal order $n^{-1/d}$, see e.g. \cite[Theorem 9.2]{GL00}, whereas the largest hole is on average of size $n^{-1/d}(\log n)^{1/d}$ and thus slightly larger than optimal, essentially due to  the coupon collectors' problem, see e.g. \cite[Corollary 2.3]{RS16}. This provides an explanation for Theorem~\ref{thm:sob-ran}.

It is natural to end this section with the following questions:

\begin{OP}
What can be used in place of $\dist(\cdot, \Pn)$ 
to derive a ``geometric'' characterization of good point sets for other classes $F$, such as unit balls in anisotropic Sobolev spaces? 
\end{OP}

\bigskip

We now turn to random information on $\Lambda^\all$, which does not have a limitation in the sense of optimal information.

\bigskip
\goodbreak

\subsection{Gaussian information}\label{sec:gauss-lin}

A geometric problem that was actually the starting point of the 
renewed interest in random information in the IBC community, 
see~\cite{HKNPU20,HKNPU21}, is the classical problem of recovering vectors from a 
symmetric convex body (a compactum with nonempty interior) $K\subset \R^m$ in the norm of $\ell_2^m$ by using $n$ linear measurements $\ell_1,\dots,\ell_n$ with $n$ much smaller than $m$. 
This fits the above setting by choosing $F=K$, which is the unit ball of a normed space $(\IR^m,\|\cdot\|_K)$, $G=\ell_2^m$ and $\Lambda=\{\ipr{\cdot,y}\colon y\in\R^m\}$. 
(Note that we can consider vectors as functions on a discrete set.)

In this case, the radius of information $r((\ell_i)_{i=1}^n,K,\ell_2^m)$, see~\eqref{eq:error-N}, has a geometric interpretation since it is equal up to a factor of 2 to the radius
\[
\rad(K\cap E_n)
:= \sup_{x \in K \cap E_n} \norm{x}_2,
\]
of $K$ intersected with the subspace 
$E_n=\{x\in\R^m\colon \ell_1(x)=\dots=\ell_n(x)=0\}$, see e.g. \cite[Lem. 4.3]{NoWo08}. If the measurements are linearly independent, then $E_n$ is of codimension $n<m$, i.e., of dimension $m-n$, 
and the smallest possible radius corresponds to the Gelfand numbers/width $c_{n}(K,\ell_2^m)$, 
see~\eqref{eq:cn}. 

It is natural to ask:
\begin{center}
How large is a ``typical'' intersection, if we choose \\
the subspace $E_n$ uniformly at random?	
\end{center}
A canonical choice of a uniform distribution is the normalized Haar measure on the set of all subspaces of codimension $n<m$, i.e. on the Grassmannian manifold $\mathcal{G}_{n,m}$.
It turns out that if 
$N_n$ is a Gaussian matrix with independent standard Gaussian entries, that is if we choose the standard Gaussian measure $\gamma_m$ on $\IR^m$,
then $E_n^{\rm ran}=\ker N_n$ is distributed according to this measure, 
and that is why we focus on this \emph{Gaussian information}. The radius of the intersection of a convex body $K\subset \IR^m$ with such a random subspace therefore satisfies
\begin{equation}\label{eq:random-rad}
	\rad(K\cap E_n^{\rm ran})
	\,\asymp\, e_n^{iid}(K,\ell_2^m,\gamma_m),    
\end{equation}
where the implicit constant is independent of any realization.

The above is a classical and well-studied question, which was tackled by many authors, 
especially for $n$ of the order $m$, i.e., intersections of large codimension.  
See e.g.~the classical results~\cite{GM98,GM97,LT00}, 
or the recent findings~\cite{GMT05,LPT06} which were obtained in the context of asymptotic geometric analysis, see \cite{AGM15} for additional references.

\medskip

In the following we apply the above results to ellipsoids of the form
\[
\mathcal{E}_\sigma = 
\left\{ (x_1,x_2,\dots) \in \ell_2 \colon \sum_{k\in \IN}\sigma_k^{-2}x_k^2 \leq 1 \right\}
\]
with square-summable semi-axes $\sigma_1\ge\sigma_2\ge \cdots\ge 0$, i.e., $\sum \sigma_j^2<\infty$. In this case $\mathcal{E}_{\sigma}$ is the unit ball of a separable Hilbert space $H\hookrightarrow \ell_2$. Note that this includes the finite-dimensional case, where we set $\sigma_k=0$ for $k>m$ and demand then that $x_k=0$.

Let $\{e_1,e_2,\dots\}\subset \ell_2$ be the standard basis and $g_1,g_2,\dots$ be iid standard Gaussian random variables. The sequence $(g_1,g_2,\dots)$ is distributed on $\R^\N$ according to the countable product of standard Gaussian measure. We define a Gaussian random functional by 
$f\mapsto \ell(f)=\sum_{j=1}^{\infty}\langle f,e_j\rangle_{2} g_j$ for $f\in H$ which almost surely absolutely converges and is therefore in $H'$. Then the restriction $\gamma$ of the distribution of $\ell$ to $H'$ is a centered Gaussian measure on $H'$ and \eqref{eq:info-ass} holds. Gaussian information is universal in the sense that it is invariant under change of basis, i.e., we have $\ell(f)=\sum_{j=1}^{\infty}\ipr{f,u_j}_2 g_j$ in distribution for any other orthonormal basis $\{u_j\}$ and $f\in H$. We refer to \cite{Bog98} for details.

In order to apply Theorem~\ref{thm:general-H}, note that for every $f,h\in H$,
\begin{align*}
\int_{H'} \langle f,g\rangle_2\langle h,g\rangle_2\dint\gamma(g)
&=\E\Big(\sum_{j=1}^{\infty}\langle f,e_j\rangle_{2} g_j\Big)\Big(\sum_{j=1}^{\infty}\langle h,e_j\rangle_{2} g_j\Big)\\
&=\sum_{j=1}^{\infty}\langle f,e_j\rangle_{2}\langle h,e_j\rangle_2
=\langle f,h\rangle_2
\end{align*}
and we get a bound on $e_N^{iid}(H,L_2,\rho_{32n}\cdot\dint\nu)$ with $N\gtrsim n\log n$ whp. For Theorem~\ref{thm:general-H}, we choose 
the random functionals $\ell(\cdot)=\sum_{j=1}^{\infty}\langle \cdot,e_j\rangle_{2} v_j$ by choosing 
the coefficients $(v_k)_{k\in\N}$ 
w.r.t.~the density $\rho_n \dint\gamma$, where 
\[
\rho_n(\ell)
=\frac{1}{2}\Big(\frac{1}{n}\sum_{k\le n}|v_k|^2+\sum_{k>n}\beta_k |v_k|^2\Big)\quad \text{with }\beta_k=\frac{\sigma_k^2}{\sum_{k>n}\sigma_k^2}. 
\]
This density concentrates around 1 (with respect to $\gamma$) as $n\to \infty$, see e.g. \cite[Lem. 1]{LM00}. 
In order to apply Lemma~\ref{lemma:concentration} we need an almost sure bound and so we cannot use density equal to one which would correspond to the geometric setting. In the following, we will see that we can choose a constant density and remove the logarithmic oversampling.

Recall that for $\ell_1,\dots,\ell_N\overset{\rm iid}{\sim} \gamma$ and $f\in H$ we have $\ell_i(f)=\sum_{j=1}^{\infty}\langle f,e_j\rangle_{2} g_{ij}$ for iid standard Gaussians $g_{ij}$ and thus $\ell_i(e_k)=g_{ik}$, i.e., the matrices in \eqref{eq:sing-1} and \eqref{eq:sing-2} will be structured Gaussian. Following \cite{HKNPU21}, we use matrix concentration bounds for these random matrices and in the following lemma combine the lower bound on the $n$-th singular value from \cite[Thm. II.13]{DS01} and the upper bound on the first from \cite[Cor. 3.11]{BH16} for simplicity in the special case $N=2n$.

\begin{lemma} \label{lemma:concentration-gaussian}
	Let $n\in\IN$ and $N=2n$. Consider $\sigma_1\ge \sigma_2\ge\cdots\ge 0$ with $\sigma\in \ell_2$. There exists $c>0$ such that with probability $1-2e^{-cn}$ we have
\[
s_n\Big(\big(N^{-1/2}\, g_{ij} \big)_{1\le i \le N,1\le j\le n}\Big)
\,\geq\, \frac{1}{2}
\]
and
\[
s_1\Big(\big(N^{-1/2}\, \sigma_j g_{ij} \big)_{1\le i \le N,j>n}\Big)
\,\leq\, 2\sqrt{\frac{1}{n}\sum_{j>n}\sigma_j^2}+2\sigma_{n+1}.
\]
\end{lemma}

\medskip
\goodbreak

Combining Lemma~\ref{lemma:concentration-gaussian} with Proposition~\ref{prop:LS-H}, 
see also~\eqref{eq:sing-1} and~\eqref{eq:sing-2}, implies that with probability $1-2e^{-cn}$ the least squares algorithm $A_N$ in \eqref{eq:alg} using $N=2n$ Gaussian random functionals $\ell_1,\dots,\ell_{N}$ and weights $1/N$ has error on $f\in H$ bounded by
\begin{equation}\label{eq:gaussian-local}
\|f-A_{N}(f)\|_{\ell_2}
\;\le\; 5\Big(\sigma_{n+1}+\sqrt{\frac{1}{n}\sum_{j>n}\sigma_j^2}\Big)\, \|f-P_nf\|_H.
\end{equation}
Due to $\sigma_{n+1}+\sqrt{\frac{1}{n}\sum_{j>n}\sigma_j^2}\lesssim \sqrt{\frac{1}{n}\sum_{j>n/2}\sigma_j^2}$ we deduce the following result obtained in \cite[Thm. 3]{HKNPU21}.

\medskip

\begin{thm} \label{thm:gauss}
There are absolute constants $b,c\in\N$ such that,  
for all $\sigma\in \ell_2$,  
we have that 
$bn$ Gaussian measurements satisfy
\[
e_{bn}^{iid}(\mathcal{E}_{\sigma},\ell_2,\gamma)
\, \le\, 
\sqrt{\frac{1}{n}\sum_{j>n}\sigma_j^2}
\]     
with probability $1-e^{-cn}$.
\end{thm}

\medskip

In \cite{HKNPU21} even a lower bound was shown such that for $\sigma_k\asymp n^{-\alpha}(\log n)^{-\beta}$ with $\alpha>0$ and $\beta\in\IR$ the characterization 
 \[
  e_n^{iid}(\mathcal{E}_{\sigma},\ell_2,\gamma)
 \;\asymp_{\alpha,\beta}\; 
  \begin{cases}
  	\sigma_1 	
        &
				\text{if} \quad \alpha < 1/2 \text{ or }\beta\le \alpha=1/2,
        \vspace*{2mm}
        \\
				\sigma_n\sqrt{\log n} 	
        &
        \text{if} \quad \beta > \alpha=1/2,
        \vspace*{2mm}
        \\
       \sigma_{n}  
        &
        \text{if} \quad \alpha>1/2,
  \end{cases}
 \]
 holds with high probability, see the proof of Corollary~7 in~\cite{HKNPU21}.  It turns out that the $n$-th minimal error of iid Gaussian information
 is of the same order as the minimal radius if $\sigma\in\ell_2$, 
while 
random information seems useless if 
$\sigma\notin\ell_2$.
Note that this is exactly the threshold we have seen for (random) 
standard information, see Section~\ref{sec:std}.

Thus, in the case of Hilbert spaces we have a complete picture of the power Gaussian information for $\ell_2$-approximation, at least for poly-logarithmic decay, and the corresponding algorithms are linear. 

In contrast, not much is known about the case of $\ell_p$-approximation.

\begin{OP}
Investigate $e_n^{iid}(\mathcal{E}_{\sigma},\ell_p,\gamma)$ for $p\neq2$, or approximation in more general norms.
\end{OP}
\medskip

In the following, we briefly discuss implications for non-Hilbert spaces. For simplicity, we restrict ourselves to the better known finite-dimensional case.

For the unit ball $K$ of a normed space $(\IR^m,\|\cdot\|_K)$, \hyperlink{assumA}{Assumption A} is satisfied if $\mu$ is the counting measure on $\{1,\dots,m\}$ and $\nu$ the standard Gaussian measure on $\IR^m$. By Lemma~\ref{lemma:emb} we find a suitable Hilbert space in $\IR^m$ such that for each $n\in 8\IN$ and $x\in K$ we have 
\[
\|x-P_nx\|_H \;\lesssim\, \sum_{k=\lfloor n/8\rfloor+1}^{m}\frac{d_k(K,\ell_2^m)}{\sqrt{k}}.
\]
If we insert this into \eqref{eq:gaussian-local}, then we obtain (similarly to the proof of Theorem~\ref{thm:general-G}) that with some constant oversampling factor $b>1$ the unweighted least-squares algorithm $A_N$ based on $N=bn$ iid Gaussian random functionals satisfies
\[
\|x-A_{N}(x)\|_2 \;\lesssim\; \frac{1}{\sqrt{n}}\sum_{k=n+1}^m\frac{d_k(K,\ell_2^m)}{\sqrt{k}},
\]
for every $x\in \IR^m$ with probability $1-2e^{-cn}$, where $c>0$ is as in Lemma~\ref{lemma:concentration-gaussian}. Thus, we get the following upper bound.
\begin{coro}
Let $K\subset \IR^m$ be a convex body. There exist $c,C>0$ such that 
\[
e_n^{iid}(K,\ell_2^m,\gamma_m)
\;\le\;\frac{C}{\sqrt{n}}\sum_{k=\lfloor cn\rfloor}^m\frac{d_k(K,\ell_2^m)}{\sqrt{k}}
\]
holds with probability $1-2e^{-cn}$ for all $n\le m$.
\end{coro}

Via \eqref{eq:random-rad} this upper bound holds for the radius of a typical section. There is a similar bound in terms of Gelfand widths, which correspond to optimal sections and therefore seems  better suited as a benchmark.

\begin{prop}[{\cite[Thm.~3.2]{LT00}}]\label{prop:radius-bound}
Let $K\subset \IR^m$ be a convex body. There exist $c,C>0$ such that 
\begin{equation}\label{eq:lt-bound}
e_n^{iid}(K,\ell_2^m,\gamma_m)
\;\le\;  \frac{C}{\sqrt{n}}\sum_{k=\lfloor cn\rfloor}^m \frac{c_k(K,\ell_2^m)}{\sqrt{k}},
\end{equation}
holds with probability $1-e^{-cn}$ for all $n\le m$.
\end{prop}
The proof of this result relies on a ``rounding technique'' together with an $M^*$-estimate, which is also employed for example in \cite{GM98,GM97,HKNPU21}. It gives a direct estimate on the radius $\rad(K\cap E_n^{\rm ran})$ instead of providing an explicit reconstruction algorithm using Gaussian information. Thus, only an abstract nonlinear algorithm can be given which matches the bound \eqref{eq:lt-bound}, see Section~\ref{sec:nonlinear}.

Using the asymptotics stated in Section~\ref{sec:sharpness} we can derive that Gaussian information is asymptotically optimal 
if the Gelfand widths decay a little faster than $n^{-1/2}$, i.e., the bodies are sufficiently ``thin'' as the dimension increases. 
It would be interesting whether this threshold of $n^{-1/2}$ 
is sharp.  
Again, this is the threshold we have seen for (random) 
standard information, see Section~\ref{sec:std}.

\begin{OP}
Is there some $K\subset\ell_2$ 
such that 
$(c_n(K,\ell_2))\notin\ell_2$, but still 
$e_n^{iid}(K,\ell_2,\gamma)\to0$ a.s.? 
\end{OP}

In order to investigate the sharpness of the bound \eqref{eq:lt-bound}, in \cite{HPS23} random sections of $\ell_p$-ellipsoids have been studied which are images of $\ell_p$-balls with $0< p\le \infty$ under diagonal operators. In the case $1<p\le \infty$ the logarithmic gaps present for poly-log decay with $\alpha=1/2$ can be narrowed. It would be interesting to do this also for general convex bodies. The proofs behind build on the same techniques used in \cite{HKNPU21} and  \cite{LT00}, and consequently yield a nonlinear algorithm.

\begin{rem}[More general linear information]
In the finite-dimensional case with a symmetric convex body $K\subset \IR^m$  which corresponds to the geometric problem of finding small sections of $K$ we note that the obtained general results not only hold for the Gaussian measure. In fact, to satisfy \hyperlink{assumA}{Assumption~(A.3)}, we can take any measure $\nu$ on $\IR^m$ which is isotropic in the sense of
\[
\int_{\IR^m}\langle x,u\rangle \langle y,u\rangle \dint\nu(u)=\langle x,y\rangle.
\]
If additionally $\nu$ has barycenter at the origin, then this corresponds to isotropicity as used in asymptotic geometric analysis, see e.g. \cite[Sec. 10.2]{AGM15}. Further note that Lemma~\ref{lemma:concentration-gaussian}, which works without logarithmic oversampling, also holds for example for Rademacher random variables instead of standard Gaussian ones, see \cite{BH16} and \cite{RV09}, and thus has implications for other types of information.
\end{rem}

\bigskip
\goodbreak
\newpage

\section{Further topics}\label{sec:further}

Let us shortly touch upon some topics close to the 
scope of this survey.

\subsection{Subsampling and optimal information}
\label{subsec:sub}

Instead of considering random sampling, and related minimal errors, 
it is clearly of interest to find relations between the different 
``benchmarks'' of optimal approximation, see Section~\ref{sec:benchmarks}. 
One particular reason is the study 
of the \emph{power} of certain classes of information, 
in which case the ``best'' information has to be considered.

It is to some extent surprising that the general results 
of the previous sections 
already lead to an optimal comparison in some cases. 
In fact, one can 
employ a subsampling technique 
based on
the famous solution to the Kadison-Singer problem \cite{MSS15}
to reduce a given ``good'' set of information to an ``optimal'' 
subset. 
We will discuss the essential lemma at the end of this section. 

This has been done in~\cite{DKU23} in the case of function values, 
see also \cite{NSU22,Te20}, and the following theorem is a slight generalization.

\goodbreak

\begin{thm}
\label{thm:sub-general}
There is a constant $b\in\N$ such that for any  
$F$, $D$, $\mu$, $\Lambda$, $\nu$ that satisfy \hyperlink{assumA}{Assumption~A}
and all $n\in \N$, we have
 \[
e_{bn}(F,L_2,\Lambda)  \,\le\, 
 \frac{1}{\sqrt{n}} \sum_{k\ge n} \frac{d_k(F,L_2)}{\sqrt{k}}.
 \]
\end{thm}

\bigskip

The bound is achieved by the corresponding algorithm~$A_N$ from 
Theorem~\ref{thm:general-H}, with the $N\asymp n\log n$ random functionals replaced by a suitable subset of order~$n$.
Again, a slight improvement is possible for Hilbert spaces. 
We omit the details and refer to~\cite{DKU23}.

\goodbreak

\goodbreak

\medskip

Since $d_k(F,L_2)=a_k(F,L_2)$ 
for all 
$F\subset L_2$,
Theorem~\ref{thm:sub-general} shows that, whenever the approximation numbers of $F$ decay at 
a polynomial rate larger 1/2, then 
information which satisfies \hyperlink{assumA}{Assumption~(A.3)} is asymptotically
as powerful as arbitrary linear information 
for $L_2$-approximation in $F$, 
at least if we only allow linear algorithms, see~\eqref{eq:an}. 

If $F=B_H$ is the unit ball of a Hilbert space, 
then it is known that $a_n(B_H,L_2)=c_n(B_H,L_2)$, 
and the last result 
shows that 
\begin{equation} \label{eq:opt}
e_{n}(B_H,L_2,\Lambda) \;\asymp\; c_n(B_H,L_2)
\end{equation}
whenever $c_n(B_H,L_2)\asymp n^{-\alpha}$ for some $\alpha>1/2$.

That is, the class $\Lambda$ contains \emph{optimal information}. Recall that this applies, e.g., 
to~$\Lambda^\all$,  
to coefficients w.r.t.~an arbitrary ONB of $L_2$, 
or to function evaluations~$\Lambda^\std$.
The latter case in particularly interesting as it was an open problem for a while. 
This, 
and the corresponding open problems from \cite{DTU16,NoWo12},
were solved in~\cite{DKU23}, 
see also~\cite{KWW09,NW11,WW01} for earlier results on this, 
and~\cite{KU21,KU21a,NSU22,U20} for direct predecessors. It is not clear what makes $\Lambda^\std$, or other information with \hyperlink{assumA}{(A.3)}, special in this context. 

This motivates the following open problem.

\begin{OP}
Find necessary and sufficient conditions on $H$ and $\Lambda$, independent of $n$, such that 
\eqref{eq:opt} 
holds true.
\end{OP}

\medskip

So far, 
we know that (A.3) together with some decay of $(c_n)$ is sufficient, 
and that $\sup_{\ell\in \Lambda}|\ell(f)|>0$ for all $f\neq0$ is necessary.

Recall that for $\Lambda^\std$
relation \eqref{eq:opt} is, in general, 
not true for Hilbert spaces with $(c_n)\notin\ell_2$, see~\cite{HNV08,KV23}.
In contrast, it is obvious by definition that no condition on the decay of $(c_n)$ is needed to achieve~\eqref{eq:opt} for $\Lambda^\all$. 
It would be interesting to find a class $\Lambda$ such that we have for some $p^*>2$, 
that \eqref{eq:opt} holds 
for Hilbert spaces $H$ with 
$\Lambda\subset H'$ 
and 
$(c_n)\in\ell_{p}$ with $p>p^*$, 
but does not hold 
for some $H$ with $(c_n)\in\ell_{p^*}$. 
For example, we do not know the answer if 
$\Lambda=\Lambda^{\rm coef}$ consists of coefficients w.r.t.~an arbitrary 
ONB of $L_2$. 
The same questions are clearly of large interest for non-Hilbert spaces, 
and approximation in other norms. 
We leave that for future research, and just note that the same subsampling approach was used in~\cite{GW24,KPUU23} to obtain the optimal bound 
\[
g_{bn}^\lin(B_H,L_\infty)
\;\asymp\; c_n(B_H,L_\infty)
\]
for all RKHS $H$ that satisfy the conditions of Theorem~\ref{thm:uniform-H}.

\bigskip

The final ingredient for the proof of Theorem~\ref{thm:sub-general} was the following infinite-dimensional version of the subsampling (or \emph{sparsification}) theorem, which allowed for direct application in the above described setting, see Proposition~17 of \cite{DKU23}.

\medskip
\goodbreak

\begin{lemma}[\cite{DKU23}]
\label{lem:subsampling}
There are absolute constants $c_1\leq 43200$, $c_2\geq 50$, $c_3\leq 21600$, 
with the following properties. 
Let $n,N\in\N$ and $y_1,\dots,y_N$ be vectors from $\ell_2(\N_0)$ 
satisfying $\|y_i\|_{2}^2\leq 2n$ and 
\begin{equation*}
\norm{\frac{1}{N}\sum_{i=1}^N y_i y_i^* 
- \begin{pmatrix} I_n &0\\0&\Lambda\end{pmatrix}}_{2\to2}
\le \frac12,
\end{equation*}
with the identity $I_n\in\C^{n\times n}$ 
and some Hermitian matrix $\Lambda$ with $\|\Lambda\|_{2\to 2}\leq1$.\\
Then, there is a subset $J\subset\{1,\dots,N\}$ 
with $|J|\le c_1 n$, such that
\[
\bigg(\frac{1}{n}\sum_{i\in J} y_i y_i^* \bigg)_{< n}
\,\ge\, c_2\,I_n
\qquad\text{and}\qquad
\frac{1}{n}\sum_{i\in J} y_i y_i^* 
\,\leq\, c_3\,I,
\]
where $A_{<n}:=(A_{k,l})_{k,l<n}$ 
and $A\le B$ denotes the Loewner order.
\end{lemma}

\medskip 

A short look to the proof of Theorem~\ref{thm:general-H} reveals 
how to apply this lemma to reduce the number of samples 
from $n \log n$ to $n$, while preserving the spectral properties of the involved matrices. 
This result, as its finite-dimensional origin, 
is fascinating, especially because it does not depend on the initial sample 
size $N$. 
See also \cite{LT22} for an application to 
sampling discretization, or \cite{KKLT22} for a survey, 
and~\cite{FS19} for the discretization of continuous frames.

Finally, we remark that Lemma~\ref{lem:subsampling} 
is ultimately due to the solution of the Kadison-Singer problem in~\cite{MSS15}, 
together with the iterative approach from~\cite{NOU16}.

\subsection{Nonlinear sampling algorithms}\label{sec:nonlinear}

Our focus is on linear algorithms but here we want to mention some results regarding nonlinear algorithms using iid information. We refer to the survey \cite{Dev98}, 
or the recent works~\cite{BCD+17,CDD09,CDPW22} and references therein, for more information on nonlinear approximation.

In general, if $H$ and $G$ are normed spaces and information is given by a map $N_n\colon H\to\IR^n$, then the reconstruction mapping 
\[
\phi^*(y):= \argmin\limits_{g\in G}\, \sup\{\|g-h\|_G\colon h\in F,\, N_n(h)=y\},
\] 
if it exists, is optimal, i.e., it attains the infimum in \eqref{eq:error-N}. 
In fact, it 
returns a Chebyshev center of the set $N_n^{-1}(y)\cap F$ considered as a subset of $G$. Composed with $N_n$ this gives an optimal nonlinear algorithm~$A_n^*$. As mentioned in Section~\ref{sec:benchmarks} and seen in Section~\ref{sec:improve}, often linear algorithms using iid information can be asymptotically as good but this is not always the case.

One of the most prominent instances of the success of iid information for nonlinear approximation is the case $F=\ell_1^m$, $G=\ell_2^m$ and $F$ the unit ball of $\ell_1^m$ which is a special case of the geometric problem mentioned in Section~\ref{sec:gauss-lin}.  
This case is related to sparse recovery, 
see e.g.~\cite{Dono06,FR12},
and was resolved already in~\cite{Ka77} 
and~\cite{GG1984}. 
In fact, it was shown that
\begin{equation}
	\label{eq:ell1-ell2}
\E\, e_n^\iid (\ell_1^m,\ell_2^m,\gamma_m) 
\;\asymp\; c_n(\ell_1^m,\ell_2^m)
\;\asymp\; \min\bigg\{1,\sqrt{\frac{\log(1+\frac{m}{n})}{n}} \,\bigg\},
\end{equation}
where the hidden constants are absolute and $\gamma_m$ denotes the standard Gaussian measure on $\R^m$.  
Note that the corresponding approximation numbers are much larger and thus nonlinear reconstructions are strictly better, see~\cite{Pi74}.

We refer to~\cite{HKNPU20} for more details on the proof of \eqref{eq:ell1-ell2} which is based on $\ell_1$-minimization or \emph{basis pursuit}
and references to generalizations. Recently, this has been generalized to $\ell_p$-ellipsoids with implications for Gelfand numbers of diagonal operators, see \cite{HPS23}. 

It is remarkable that so far, despite its enormous importance for applications, 
there is no explicit, deterministic construction of a 
\emph{near-optimal} $N_n$ attaining the upper bound in \eqref{eq:ell1-ell2}. 
The same is true for several of the results from Section~\ref{sec:general}.

\begin{rem}
    Let us note that in the original bound in~\cite{Ka77} on $e_n^\iid(\ell_1,\ell_2,\gamma_m)$ the exponent of $\log(1+\frac{m}{n})$ in \eqref{eq:ell1-ell2} is $3/2$ instead of $1/2$. It is somehow interesting that, given the optimal bound on $c_n(\ell_1^m,\ell_2^m)$ in \eqref{eq:ell1-ell2} this can be obtained from the bound in \cite{LT00} presented in Proposition~\ref{prop:radius-bound}.
\end{rem}

In the context of sampling numbers, we want to mention further recent results based on  sparse approximation and iid random points:

Using a greedy (and nonlinear) approximation method, as well as iid uniform random points on rather general domains \cite{DT24} obtained bounds for $L_2$-approximation in general function classes, thereby improving upon recent results in \cite{JUV23} obtained via \emph{basis pursuit denoising}, another nonlinear reconstruction method. 
See also~\cite{Kri23} for an analysis of this method with emphasis on high-dimensional approximation.

\subsection{Randomized algorithms}

Randomized algorithms, also known as~\emph{Monte Carlo methods}, 
are a larger class of algorithms which, 
in contrast to the algorithms discussed so far, 
are allowed to use different information for each input, and additional random numbers. 
(Although we studied random information, 
we considered the deterministic worst-case error for each realization as in~\eqref{eq:error-N}, and hence, do not allow random algorithms in this sense.)
That is, a Monte Carlo method $M_n$ is a random variable, 
that depends in expectation on $n$ pieces of information of the input, and
we define the \emph{worst-case (root-mean-square) error} 
\begin{equation*}
e^{\mathrm{ran}}(M_n,F,G) 
\;:=\; \sup_{f\in F}\, \sqrt{\E\Bigl[\norm{f-M_n(f)}_G^2\Bigr]}
\end{equation*}
as well as the \emph{$n$-th minimal randomized errors} 
$e_n^{\mathrm{ran}}(F,G,\Lambda)$ 
as the infimum over all such methods. 
We clearly have $e^{\mathrm{ran}}_n(F,G,\Lambda)\le e_n(F,G,\Lambda)$, 
because every deterministic algorithm can be considered 
a (constant) random variable.
In addition, randomized methods might be quite advantageous
and more generally applicable. 
However, such methods do usually not allow for reliable error guarantees, in the sense that
error bounds only hold with certain probability, 
and that a realization of a randomized algorithm 
may have small error for some $f\in F$, but not for all at once.

There are even many situations 
(e.g., if $H$ and $G$ are Hilbert spaces, $F$ is the unit ball of $H$, 
and we allow arbitrary linear information)
where randomness does not help at all compared to deterministic algorithms. 
We refer to~\cite{H94,No88,NoWo08,NoWo10} for more details 
and general results on randomized approximation.

\medskip

We will see below that a randomized least squares method can attain the optimal results under much weaker conditions.
First, it is intuitively clear, and can be seen from the corresponding proofs, 
that the ``discretization condition''~\eqref{eq:prop-cond}, or~\eqref{eq:matrix-1}, is crucial also in this case. 
But the ``stability condition''~\eqref{eq:cond2}, or~\eqref{eq:matrix-2}, which depends on the 
class $F$, can be weakened to
\begin{equation} \label{eq:cond-MC}
\PP\left(
\sum_{i=1}^N w_i |\ell_i(g)|^2
\;\le\; \beta^2\,\|g\|_{L_2}^2
\right) \,\ge\, 1-\eta 
\qquad \text{ for all }\; g\in L_2,
\end{equation}
 for some $\eta>0$, 
where $w_i$ and $\ell_i$ are the random weights and functionals, respectively. 
This condition is easy to verify for iid $\ell_i$ if they are sampled, as before, with respect to a density $\rho_n'$, and we set $w_i=1/\rho_n'(\ell_i)$.

The following has been obtained essentially in~\cite{CM17}.

\begin{prop}
Let $V_n\subset L_2$ be an $n$-dimensional space, 
and 
$w_i$, $\ell_i$ be such that~\eqref{eq:cond-MC} and 
\begin{equation} \label{eq:cond-MC-2}
\PP\left(
\sum_{i=1}^N w_i |\ell_i(g)|^2 
\;\ge\; \alpha^2\,\|g\|_{L_2}^2
\quad \text{ for all }\; g\in V_n 
\right) \,\ge\, 1-\delta 
\end{equation}
hold for some $\alpha,\beta,\eta,\delta>0$. Then,  the algorithm $A_N$ from~\eqref{eq:alg} satisfies for each $f\in L_2$ that 
\[
\Big\|f - A_N(f)\Big\|_{L_2}^2
\;\le\; \left(1+\frac{\beta}{\alpha}\right)\, d(f,V_n,L_2)^2
\]
with probability $1-\eta-\delta$, 
where $d(f,V_n,G):=\min_{g\in V_n} \big\Vert f - g\big\Vert_{G}$.
\end{prop}

\medskip

We see that, once the conditions are verified, 
we obtain a near-optimal approximation in arbitrary subspaces of $L_2$. 
We do not need to assume that $f\in F$ for some class $F$ with decaying widths. 
It is clear that a result of this kind cannot be true in a deterministic setting, or with probability one.

To obtain a bound in expectation, i.e., on the error $e^\ran(A_N)$, we need to 
control the error for 
realizations of $A_N$ for which~\eqref{eq:cond-MC} and~\eqref{eq:cond-MC-2} do not hold. 
This can be done in different ways. 
In~\cite{CM17}, where this result was 
applied first to $N\asymp n \log n$ iid points, 
the authors proceeded by considering an error bound in terms of $d(f,V_n,L_\infty)$, or by adding a term $n^{-r}\|f\|_2$ on the right hand side, where $N$ must grow with $r$, see also~\cite{CCMNT15,CDL13}. 
The required sampling density $\rho'_n$ is the first summand of $\rho_n$ in~\eqref{eq:density}. 
In~\cite{HNP22}, the algorithm was analyzed 
for iid random points distributed according to this density, 
conditioned on the event in~\eqref{eq:cond-MC-2}. Since~\eqref{eq:cond-MC} also holds in expectation, one obtains 
$\E\|f - A_N(f)\|_{L_2}^2
\lesssim d(f,V_n,L_2)^2$ 
for $N\asymp n \log n$. 
Based on a similar subsampling 
idea as discussed in Section~\ref{subsec:sub}, 
this led to the 
important result from~\cite{CD22} which shows that, in the case of $L_2$-approximation, linear randomized algorithms based on function values can be optimal among arbitrary linear algorithms. See~\cite{K-MC,WW07} for earlier results, and \cite{ChDo23} for a recent refinement leading to explicit and  smaller constants and oversampling.

\begin{thm}[\cite{CD22}]
There exist constants $b,C\in\N$ such that the following holds. \\
For any $n$-dimensional space $V_n\subset L_2(D,\mu)$, there 
is a random variable $X$ on $\binom{D}{bn}$, i.e., all $(bn)$-element subsets of $D$, such that the algorithm 
$A_N$ from~\eqref{eq:alg} 
with $N=bn$, $\{x_1,\dots,x_{bn}\}=X$ and $w_i:=\min_{v\in V_n}\frac{\|v\|_{L_2}^2}{|v(x_i)|^2}$ satisfies 
\[
\E\Big\|f - A_N(f)\Big\|_{L_2}^2
\;\le\; C\, \min_{g\in V_n} \big\Vert f - g\big\Vert_{L_2}^2
\qquad \text{for all }\; f\in L_2.
\]
In particular, 
\[
e^{\mathrm{ran}}_{bn}(F,L_2,\Lambda^{\std}) 
\;\le\;
C\cdot a_n(F,L_2)
\]
for any compact subset $F\subset L_2$. 
\end{thm}

\medskip

The random sample used by the above algorithm is 
not given by iid~random points, 
and it is again not too difficult to see (by using the coupon collector's problem) 
that such a result cannot be true for iid information in general.

\bigskip

Similar to the results and techniques from Section~\ref{sec:sobolev}, 
it has been shown in \cite{KNS22} that iid samples are asymptotically optimal in expectation 
for $L_q$-approximation of Sobolev functions in $W^s_p$ on very general domains, 
except for the case $p=q=\infty$. This is an improvement compared to the deterministic error discussed in Section~\ref{sec:sobolev}. Again, we would like to know how this generalizes.

\begin{OP}
Find conditions on $F$ such that there exists a density $\rho$ with 
$\E\|f - A_N(f)\|_{L_2}^2
\lesssim d(f,V_n,L_2)^2$ 
for all $f\in F$, where $A_N$ from~\eqref{eq:alg} is based on $N\asymp n$ iid points from $\rho$.
\end{OP}

Regarding nonlinear approximation which was discussed in Section~\ref{sec:nonlinear}, Gaussian information has proven useful also in the randomized setting for approximation between sequence spaces and $L_q$-approximation in Sobolev spaces, see e.g. \cite{Hei92,Mat91}. Recently, in \cite{KNW23}, following \cite{Hei23}, improvements using adaptivity have been shown, which is in contrast to the deterministic case, where adaption is useless for linear problems, see \cite[Thm. 4.4]{NoWo08}.

\subsection{High dimensions and tractability}
\label{sec:tractability}

Many problems in numerical approximation have an associated dimension $d$, 
that of the domain, which influences the error. 
In IBC it is of large interest to study the dependence of the minimal errors on $d$. 

In this context, it can be useful to state results in 
terms of the \emph{information complexity} 
\[
n(\eps,F,G,\Lambda) \,:=\, 
\min\bigl\{n\in\N\colon\, e_n(F,G,\Lambda)\le\eps\bigr\},
\]
which is the minimal number of values of information functionals from $\Lambda$ 
needed to achieve an error of at most $\eps>0$. 
That is, every algorithm that achieves an error in $G$ of at most $\eps>0$ 
for all $f\in F$ needs at least $n(\eps,F,G,\Lambda)$ pieces of information from $\Lambda$. 

To classify the \emph{difficulty} of a problem in higher dimensions 
we consider a sequence of function classes $F_d$ on domains $D_d$
(and associated $G_d$ and $\Lambda_d$), $d\in\N$. 
A problem (or a sequence thereof) is then called
\emph{polynomially tractable}, if there exist absolute constant 
$\alpha,\beta,C\ge0$, 
such that
\[
n(\eps,F_d,G_d,\Lambda_d) \;\le\; C\;d^\alpha\;\eps^{-\beta}
\qquad \text{ for all } \eps>0 \;\text{ and }\; d\in\N,
\]
i.e., the needed 
amount of information 
depends at most polynomially on $d$ and $1/\eps$. 
In contrast, a problem is said to suffer from the \emph{curse of dimension}, 
if there exist absolute constants $\eps_0,d_0,\gamma,C>0$, such that
\[
n(\eps,F_d,G_d,\Lambda_d) \;\ge\; C\;(1+\gamma)^d
\qquad \text{ for all } \eps\in(0,\eps_0) \;\text{ and }\; d\ge d_0.
\]
That is, one needs exponentially many pieces of information to find an approximate 
solution to the problem.
Such a problem is generally assumed to be \emph{intractable}. 
For a comprehensive account on tractability and the concepts mentioned in the remainder of this section we refer to the books~\cite{NoWo08,NoWo10,NoWo12}.

\begin{rem}
The term \emph{``curse of dimension''} 
has been introduced by Bellmann~\cite{Bell57} in 1957 for the 
phenomenon that the number of needed samples increases exponentially 
with the (input) dimension.
This is very much inspired by classifications in discrete complexity theory and we use the same concept.
Note, however, that this term gained prominence in several other areas and is
sometimes used for saying that 
the order of convergence decreases to zero (e.g., like $e_n\asymp_d n^{1/d}$). 
The examples below will show that this is not the same, 
and that the $d$-dependent ``constants'' are of significance. 
\end{rem}

\medskip

As a result of the structure of function spaces, in many cases the curse of dimension cannot be avoided. 
It is largely open, and one of the main concerns of IBC, to identify natural classes of functions where the curse of dimension does not hold, 
especially for $\Lambda^\std$. \\
In the context of this survey let us only mention that the spaces $W_p^s(D)$ from Section~\ref{sec:sobolev} are too large to be tractable. 
Namely, it has been shown in~\cite{HNUW14,HNUW14b,HNUW17}, 
see also~\cite{Su79} for the case $s=1$, 
that for any sequence of volume-normalized domains $(D_d)$, 
like $D_d=[0,1]^d$, already the integration problem for $W_\infty^s(D_d)$ suffers from the curse of dimension. In fact, 
\[
n(\eps,W_\infty^s(D_d),L_q,\Lambda^\std) 
\;\gtrsim_s\; \left(\frac{d}{\eps}\right)^{d/s}
\]
for all $1\le q\le\infty$, $s\in\N$ 
and volume-normalized $(D_d)$.
Even larger bounds exist for $q=\infty$, see~\cite{K19}, in which case the curse also holds for $s=\infty$, see~\cite{NoWo09}. 

As a remedy to the curse of dimension we mention the prominent example of
\emph{weighted spaces}, introduced in~\cite{SW98}, where different ``weights'' are assigned to different coordinates, 
see~\cite{NoWo08,NoWo10,NoWo12} for a discussion.
For recent progress in the context of ``unweighted'' spaces, see \cite{Kri23} where tractability of $L_q$-approximation in spaces with bounded sum of absolute values of Fourier coefficients with respect to a suitable basis was obtained. 

Note that many of the error bounds presented above come with (explicit) absolute constants and are therefore also suitable for tractabilty studies, see e.g.~\cite{GW24,KSUW23}.

\bigskip 

There is another prominent problem in IBC, which emphasizes the value of iid information in this context: 
Distributing points ``uniformly'' on $[0,1]^d$.

For this, let the (star-)discrepancy of a point set $\Pn$ be given by
\[
D(\Pn)=\;:=\; \sup_{x\in [0,1]^d} \Big|\frac{\#(\Pn\cap [0,x])}{n} - {\rm vol}([0,x])\Big|,
\]
where $[0,x]=[0,x_1]\times\cdots\times [0,x_d]$. Optimizing over all $n$-point sets  $\Pn\subset [0,1]^d$ gives the 
\emph{$n$-th minimal discrepancy} 
\vspace{-2mm}
\begin{equation*}
D(n,d) \,:=\, \inf_{\Pn}\,D(\Pn).
\end{equation*} 
This is a very important and extensively studied 
quantity in the field of \emph{irregularity of distribution}~\cite{BC87}. 
See also~\cite{DP14,Hi13,No15} for recent treatments. \\
Via the prominent \emph{Koksma-Hlawka inequality}~\cite{Hl61} the discrepancy of $\Pn$ is related to the radius of information for integration 
in a space of mixed smoothness, see e.g.~\cite{DP14}. 
To date, the best bounds (for large $n$) are 
\begin{equation}\label{eq:disc}
n^{-1}\, (\log n)^{(d-1)/2+\eta_d} 
\,\lesssim_d\, D(n,d) 
\,\lesssim_d\, n^{-1}\, (\log n)^{d-1},
\end{equation}
with some small $\eta_d>0$, see~\cite{BLV08} for the lower bound and e.g. \cite{DP14} for the upper bound. 
Regarding the dependence on $d$, there exist $c_1,c_2,\varepsilon_0>0$ such that
\begin{equation}\label{eq:disc-d}
c_1\, \min\Big\{\varepsilon_0,\frac{d}{n}\Big\} 
\,\le\, D(n,d) 
\,\le\, c_2\, \sqrt{\frac{d}{n}}\qquad \text{for all }d,n\in\IN,
\end{equation}
see~\cite{HNWW01,Hi04}.
We obtain that the number of points needed to achieve a discrepancy less than 
$\eps>0$ is (up to constants) between $d\,\eps^{-1}$ and $d\,\eps^{-2}$ 
and hence, linear in~$d$. The bound from~\eqref{eq:disc}, which increases exponentially with~$d$, is therefore not enough to conclude a statement on 
the complexity in high-dimensions.

The upper bound due to \cite{HNWW01} is achieved by iid uniform random points and relies on \cite{Tal94-emp} which employs empirical process theory and the concept of \emph{Vapnik-\v{C}ervonenkis (VC) dimension}, a notion of complexity originating from statistical learning theory. 
In high dimensions, iid points achieve also the best known bounds for the related notion of \textit{dispersion}, which measures the volume of the largest empty box, see~\cite{Lit21,UV18}.

Note that due to the central limit theorem the rate $n^{-1/2}$ in \eqref{eq:disc-d} cannot be improved 
for iid points. 
Even more, for uniformly distributed iid points there is also a lower bound of the order 
$\sqrt{d/n}$ 
for $n\gtrsim d$ which holds with exponentially large (in $d$) probability, 
see~\cite{Doe14}.

Improvements on~\eqref{eq:disc} 
and~\eqref{eq:disc-d}, and the construction of points satisfying the latter, seem to be very challenging open problems.

\begin{OP}
Is the upper bound in \eqref{eq:disc-d} sharp for all $n$ that are at most polynomially large in $d$? In other words, do iid uniform random points have optimal discrepancy in high dimension?\\ 
Moreover, find explicit deterministic constructions that satisfy 
$D(\Pn)\lesssim \frac{d^{42}}{\sqrt{n}}$. 
\end{OP}

\subsection{A machine learning perspective}

In this final section we want to provide a different point of view on the setting of this survey. In other literature, especially~from data science,
it is usually assumed that some ``data'' $(x_1,y_1),\dots,(x_N,y_N)\in \mathcal{X}\times \mathcal{Y}$ is produced by iid samples of a random vector $(X,Y)$ with distribution $\rho$ on $\mathcal{X}\times \mathcal{Y}$. 
Our setting (for standard information) amounts to $(X,Y)=(X,f(X))$ for some function $f\colon \mathcal{X}\to \mathcal{Y}$ which we would like to approximate using the given data $(x_1,f(x_1)),\dots,(x_N,f(x_N))$. 

In machine learning, 
one wants to ``explain'' the data by finding a function $f\colon \mathcal{X}\to \mathcal{Y}$ (a \emph{model}) that maps an input~$x$ to an output~$y$. 
Here, an additional (additive) noise is a typical assumption. 
Denoting by $\rho_{Y|X}$ the conditional probability distribution given $X$, the regression function
\[
f_{\rho}(x)=\int_{\mathcal{Y}}y\dint\rho_{Y|X}(y|x),\quad x\in \mathcal{X},
\]
is the best guess of $y$ given $x$ with respect to an $L_2$-error and the goal is to approximate this function using the given data. For example, 
one uses empirical best approximation in an hypothesis space $\mathcal{H}$, i.e., 
\[
\hat{f}_{\bf z} \,:=\, \argmin_{\hat{f}\in\mathcal{H}}\frac{1}{N} \sum_{i=1}^{N}|\hat{f}(x_i)-y_i|^2,
\]
where ${\bf z}:=((x_1,y_1),\dots,(x_N,y_N))$ 
is the given data. \\
If we define the (squared) \emph{error} of the model $g$ by 
\[
\cE(g) \,:=\, \cE_\rho(g) \,:=\, \int_{\mathcal{X}} |g(x)-y|^2 \dint\rho(x,y), 
\]
then it is easy to verify that the error of the least squares estimator decomposes 
into $\cE(\hat{f}_{\bf z})=\int_{\mathcal{X}} |\hat{f}_{\bf z}(x)-f_\rho(x)|^2 \dint\rho_X(x)+\int_{\mathcal{X}} |f_\rho(x)-y|^2 \dint\rho(x,y)$, where $\rho_X$ denotes the marginal of $\rho$ on $X$. 

Naively, our setting in Section~\ref{sec:general} corresponds to choosing $\mathcal{H}=V_n$ for a suitable $n$-dimensional subspace $V_n$ of $L_2$, where $n$ depends on $N$, and
$f_\rho=f$ (i.e., $y_i=f(x_i)$). 
This excludes for example the noisy case $(X,Y)=(X,f(X)+\varepsilon)$ where $\varepsilon$ is centered noise independent of $X$. 

However, there is another interpretation we would like to comment on. We only know 
the data $(x_i,y_i)_{i=1}^N\sim\rho$, 
and the goal is to compute $f_\rho$. 
Assuming all $x_i$ are different and there is no noise, 
we may assume that $y_i=y(x_i)$ for some function 
$y\colon\mathcal{X}\to\mathcal{Y}$.
Hence, we can write $\hat{f}_{\bf z}=A_N^u(y)$ with $A_N^u$ from~\eqref{eq:alg-u}.
If the data and 
hypothesis space $V_n$ is 
such that $A_N^u(g)=g$ for all $g\in V_n$, 
and $\hat{f}:= \argmin_{g\in V_n}\|f_\rho-g\|_{L_2}$, then 
\[\begin{split}
\|f_\rho-A_N(y)\|_{L_2}
\;&\le\;\|f_\rho-\hat{f}\|_{L_2}
\,+\, \|\hat{f}-A_N(y)\|_{L_2} 
\;=\;
\cE(\hat{f})^{1/2}
\,+\, \|A_N(y-\hat{f})\|_{L_2}. \\
\end{split}\]
The first term is often called the \emph{approximation error}, 
and can not be avoided due to the choice of the hypothesis space. 
The second is the \emph{sample error}, 
and depends on the ``quality'' of the given data. 
Hence, if we assume that the ``error'' $\eps=y-\hat{f}$ is not ``just random'' but has a certain structure, 
then the considerations of this survey might be of interest 
for further studies, see e.g.~Section~4 of~\cite{KKLT22}.

For 
the mathematical foundations of learning we refer to~\cite{CuSm02} and~\cite{Fou22}. 

\goodbreak

\subsection*{Acknowledgements}

We are grateful for the hospitality of the Leibniz Center for Informatics at Schloss Dagstuhl where they worked on this survey during the Dagstuhl Seminar 23351 \emph{``Algorithms and Complexity for Continuous Problems''}. Further, we thank 
Albert Cohen, Matthieu Dolbeault, David Krieg and Erich Novak   
for valuable comments on an earlier version of this survey, 
as well as for fruitful discussions on the subject in general. 
We also thank Peter Math\'e for pointing out the relation to the Bernstein numbers in Section~\ref{subsec:general}. MS has been supported by the Austrian Science Fund (FWF) Project P32405 \emph{``Asymptotic geometric analysis and applications''}.

\end{document}